\documentclass[11pt,reqno]{amsart}
\usepackage{amsmath,amssymb}
\usepackage[colorlinks=true,linkcolor=magenta,citecolor=blue]{hyperref}
\usepackage[margin=1in]{geometry}
\usepackage{enumitem}
\usepackage{mathabx}
\usepackage{epsfig}

\newtheorem{theorem}{Theorem}[section]
\newtheorem{proposition}{Proposition}[section]
\newtheorem{lemma}{Lemma}[section]

\newtheorem{remark}{Remark}[section]

\newcommand{\R}{\mathbb{R}}

\newcommand{\e}{\varepsilon}
\newcommand{\ity}{\infty}

\title[Semi-linear structurally damped wave equations with nonlinear convection]{Global existence results for semi-linear structurally damped wave equations with nonlinear convection}
\author{Tuan Anh Dao}
\address{Tuan Anh Dao \hfill\break
$\quad$ School of Applied Mathematics and Informatics, Hanoi University of Science and Technology, No.1 Dai Co Viet road, Hanoi, Vietnam \hfill\break
Institute of Mathematics, Vietnam Academy of Science and Technology, No.18 Hoang Quoc Viet road, Hanoi, Vietnam}
\email{anh.daotuan@hust.edu.vn}

\author{Hiroshi Takeda}
\address{Hiroshi Takeda \hfill\break
$\quad$ Faculty of Engineering, Fukuoka Institute of Technology \hfill\break
3-30-1 Wajiro-higashi, Higashi-ku, Fukuoka, 811-0295 Japan}
\email{h-takeda@fit.ac.jp}

\begin{document}
\subjclass[2010]{35A01, 35L15, 76E06}
\keywords{Wave equations; Structural damping; Nonlinear convection; Global existence}

\begin{abstract}
In this paper, we consider the Cauchy problem for semi-linear wave equations with structural damping term $\nu (-\Delta)^2 u_t$, 
where $\nu >0$ is a constant. 
As being mentioned in \cite{GhisiGobbinoHaraux2016, IkehataIyota}, 
the linear principal part brings both the diffusion phenomenon and the regularity loss of solutions.   
This implies that, for the nonlinear problems,   
the choice of solution spaces plays an important role to obtain the global solutions with the sharp decay properties in time. 
Our main purpose of this paper is to prove the global (in time) existence of solutions for the small data and their decay properties for the supercritical nonlinearities. 
\end{abstract}

\maketitle

\section{Introduction}
In this paper, let us study the following Cauchy problem for the semi-linear wave equation with structural damping term:
\begin{equation}
\begin{cases}
u_{tt}- \Delta u+ \nu (-\Delta)^2 u_t= a \circ \nabla f(u,u_t), &\quad x\in \R^n,\, t > 0, \\
u(0,x)= u_0(x),\quad u_t(0,x)=u_1(x), &\quad x\in \R^n,
\end{cases}
\label{equation1.1}
\end{equation}
where $\nu$ is a positive constant, $a\ne 0$ is a constant vector in $\R^{n}$ and $\circ$ represents the inner product in $\R^{n}$. 
Here we mainly deal with two kinds of the nonlinear terms:
\begin{equation*}
\begin{split}
 f(u,u_t)=
\begin{cases}
|u|^{p}, \\
|u_{t}|^{p},
\end{cases}
\end{split}
\end{equation*}
where $p>1$. 
We are interested in finding out the conditions for the growth order of the nonlinearities, the so-called admissible exponents $p$, 
which ensure the global solutions to \eqref{equation1.1} for the small data. 

The crux of our proof ideas of the main results is to apply the derived decay estimates for solutions to the corresponding linear Cauchy problem in the treatment of the nonlinear convection terms.
For this reason, here we would like to give the brief historical survey for the linear equation and some previous results related to nonlinear convection.
Namely, the corresponding linear equation of \eqref{equation1.1} is given by
\begin{equation}
\begin{cases}
u_{tt}- \Delta u+ \nu (-\Delta)^2 u_t= 0, &\quad x\in \R^n,\, t > 0, \\
u(0,x)= u_0(x),\quad u_t(0,x)=u_1(x), &\quad x\in \R^n.
\end{cases}
\label{equation1.2}
\end{equation}
As we can see, this equation was proposed by Ghisi-Gobbino-Haraux \cite{GhisiGobbinoHaraux2016} as one of the modeling case of the second order abstract evolution equation. 
In the cited paper, the authors claimed that the regularity loss type estimates appear not only for solutions $u(t,x)$ itself but also for their derivatives in time $u_t(t,x)$ to \eqref{equation1.2}.
After that, Ikehata-Iyota \cite{IkehataIyota} obtained the asymptotic profile of solutions to \eqref{equation1.2} with the weighted $L^{1,1}$ initial data in suitable space.
Moreover, they also proved that the low frequency part is dominant and the solution is approximated by the diffusion wave as $t \to \infty$.    
Quite recently, Fukushima-Ikehata-Michihisa \cite{FukushimaIkehataMichihisa} have studied the higher order asymptotic expansion of solutions to \eqref{equation1.2}, which corresponds to the additional regularity assumptions on the initial data. 
The other point worthy of mentioning is that the authors have determined a threshold of the regularity condition for the initial data to classify whether the leading factor of the asymptotic profiles of solutions as $t \to \infty$ is given by the low frequency parts or not.
Concerning the study of nonlinear convection, we want to refer the interested readers to a series of previous work, for examples, \cite{EscobedoZuazua1991,Zuazua93,AguirreEscobedo1993,EscobedoZuazua1997} in term of the convection-diffusion equation and references therein.
One should recognize that in the cited papers the large time behavior of solutions has been explored by supposing the initial data with some kind of different regularities.

To the best of the authors' knowledge, there seem not so many research papers regarding the investigation of the semi-linear equation \eqref{equation1.1} so far.
For this reason, generally speaking, previous papers suggest that it is difficult to construct sharp time decay estimates for solutions to the dissipative hyperbolic equations with derivative loss structure. 
Especially, the same situation happens when the growth order of nonlinearities is nearby the critical case even if the nonlinear function belongs to $\mathcal{C}^{2}$ (cf. \cite{D'Abbicco2017}).
We also remark that this is observed when the nonlinearities are smooth (see \cite{SugitaniKawashima}).   

\noindent\textbf{Notations:} Throughout this paper, we use the following notations. We denote $f\lesssim g$ if there exists a constant $C>0$ such that $f\le Cg$, and $f \sim g$ if $g\lesssim f\lesssim g$. We write $\widehat{h}(t,\xi):= \mathcal{F}_{x\rightarrow \xi}\big(h(t,x)\big)$ as the Fourier transform with respect to the spatial variable of a function $h(t,x)$. As usual, $H^s$ and $\dot{H}^s$, with $s \ge 0$, stand for Bessel and Riesz potential spaces based on $L^2$ spaces. For any $\sigma\ge 0$, the weighted spaces $L^{1,\sigma}$ are defined by
$$ L^{1,\sigma}= \Big\{f\in L^1 \,:\, \|f\|_{L^{1,\sigma}}:= \int_{\R^n} (1+|x|)^\sigma |f(x)|\,dx<+\ity \Big\}. $$
Moreover, if we introduce the space $\mathcal{D}:= \big(L^{1,\sigma} \cap H^{s_1}\big) \times \big(L^{1,\sigma} \cap H^{s_2}\big)$ with $\sigma,s_1,s_2\ge 0$, then the norm is defined by
$$ \|(v_0,v_1)\|_{\mathcal{D}}=\|v_0\|_{L^{1,\sigma}}+ \|v_0\|_{H^{s_1}}+ \|v_1\|_{L^{1,\sigma}}+ \|v_1\|_{H^{s_2}}. $$

Now we consider the nonlinear function $f(u,u_t)=|u|^p$ in (\ref{equation1.1}), i.e. the following semi-linear equation:
\begin{equation}
\begin{cases}
u_{tt}- \Delta u+ \nu (-\Delta)^2 u_t= a \circ \nabla |u|^p, &\quad x\in \R^n,\, t > 0, \\
u(0,x)= u_0(x),\quad u_t(0,x)=u_1(x), &\quad x\in \R^n.
\end{cases}
\label{equation3.1}
\end{equation}
Our first result states the global (in time) existence of solutions to \eqref{equation3.1} with their sharp decay properties.
\begin{theorem} \label{theorem3.1}
Let $n=2,3,4,5$ and $\e$ is a sufficiently small positive constant. 
Suppose that the following condition:
\begin{equation} \label{exponent3.1}
\begin{split}
&p> 5 \hspace{2.2cm} \text{if}\ n=2, \\
&p\ge 1+\dfrac{4}{n-1} \qquad \text{if}\ n=3,4,5.
\end{split}
\end{equation}
Then, there exists a constant $\e_0>0$ such that for any small data
$$ (u_0,u_1) \in \mathcal{A}:= \big(L^1 \cap H^{\frac{n}{2}+\e}\big) \times \big(L^1 \cap L^{2} \big) $$
satisfying the assumption $\|(u_0,u_1)\|_{\mathcal{A}}\le \e_0,$ 
the Cauchy problem \eqref{equation3.1} admits a unique global in time solution in the class 
$$ u \in \mathcal{C}\big([0,\ity),H^{\frac{n}{2}+\e}\big)\cap \mathcal{C}^1\big([0,\ity), L^{2} \big). $$ 
Moreover, the following estimates hold:
\begin{align}
\|u(t,\cdot)\|_{L^2}& \lesssim
\begin{cases}
\sqrt{\log(t+e)} \|(u_0,u_1)\|_{\mathcal{A}} &\text{ if }\, n=2, \\
(1+t)^{-\frac{n}{8}+\frac{1}{4}} \|(u_0,u_1)\|_{\mathcal{A}} &\text{ if }\, n=3,4,5,
\end{cases} \label{estimate3.1.1} \\
\big\|\nabla^{\frac{n}{2}+\e} u(t,\cdot)\big\|_{L^2}& \lesssim 
\begin{cases}
(1+t)^{-\frac{n-1+\e}{4}} \|(u_0,u_1)\|_{\mathcal{A}} &\text{ if }\, n=2,3,4, \\
(1+t)^{-\frac{3}{4}+\frac{\e}{2}} \|(u_0,u_1)\|_{\mathcal{A}} &\text{ if }\, n=5,
\end{cases} \label{estimate3.1.2} \\
\|u_t(t,\cdot)\|_{L^2}& \lesssim (1+t)^{-\frac{n}{8}} \|(u_0,u_1)\|_{\mathcal{A}} \label{estimate3.1.3}.
\end{align}
\end{theorem}
\begin{remark}
\fontshape{n}
\selectfont
Here we note that the estimate \eqref{estimate3.1.2} for $n=5$ indicates the loss of time decay. 
As a result, we cannot simply apply the proof of Theorem \ref{theorem3.1} from the cases of $n=2,3,4$ to the case of $n=5$.
Then, we need to modify the solution space to obtain the sharp decay property of the $L^{2}$ norm \eqref{estimate3.1.1} for $n=5$.
\end{remark}

Next, we study the nonlinear function $f(u,u_t)= |u_t|^p$ in (\ref{equation1.1}), i.e. the following semi-linear equation:
\begin{equation}
\begin{cases}
u_{tt}- \Delta u+ \nu (-\Delta)^2 u_t= a \circ \nabla |u_t|^p, &\quad x\in \R^n,\, t > 0, \\
u(0,x)= u_0(x),\quad u_t(0,x)=u_1(x), &\quad x\in \R^n.
\end{cases}
\label{equation4.1}
\end{equation}
In this case, we can construct global solutions with their sharp decay properties because of the gained regularity of derivative in time of solutions from the linear principal part as mentioned in \cite{GhisiGobbinoHaraux2016}.  
\begin{theorem} \label{theorem4.1}
Let $n=2,3,4$ and $\e$ is a sufficiently small positive constant. We assume the following condition:
\begin{equation}
p \ge 1+ \max\Big\{\frac{3}{n}, 1 \Big\}. \label{exponent4.1}
\end{equation}
Then, there exists a constant $\e_0>0$ such that for any small data
$$ (u_0,u_1) \in \mathcal{B}:= \big(L^1 \cap H^{\frac{n}{2}+\e}\big) \times \big(L^1 \cap H^{\frac{n}{2}+\e}\big) $$
satisfying the assumption $\|(u_0,u_1)\|_{\mathcal{B}}\le \e_0,$ we have a uniquely determined global (in time) solution
$$ u \in \mathcal{C}\big([0,\ity),H^{\frac{n}{2}+\e}\big)\cap \mathcal{C}^1\big([0,\ity),H^{\frac{n}{2}+\e}\big) $$
to \eqref{equation4.1}. Moreover, the following estimates hold:
\begin{align}
\|u(t,\cdot)\|_{L^2}& \lesssim
\begin{cases}
\sqrt{\log(t+e)} \|(u_0,u_1)\|_{\mathcal{B}} &\text{ if }\, n=2, \\
(1+t)^{-\frac{n}{8}+\frac{1}{4}} \|(u_0,u_1)\|_{\mathcal{B}} &\text{ if }\, n=3,4,
\end{cases} \label{estimate4.1.1} \\
\big\|\nabla^{\frac{n}{2}+\e} u(t,\cdot)\big\|_{L^2}& \lesssim (1+t)^{-\frac{n-1+\e}{4}} \|(u_0,u_1)\|_{\mathcal{B}}, \label{estimate4.1.2} \\
\|u_t(t,\cdot)\|_{L^2}& \lesssim (1+t)^{-\frac{n}{8}} \|(u_0,u_1)\|_{\mathcal{B}},  \label{estimate4.1.3} \\
\big\|\nabla^{\frac{n}{2}+\e} u_t(t,\cdot)\big\|_{L^2}& \lesssim (1+t)^{-\frac{n+\e}{4}} \|(u_0,u_1)\|_{\mathcal{B}}. \label{estimate4.1.4}
\end{align}
\end{theorem}
\begin{remark}
\fontshape{n}
\selectfont
It is clear that we have assumed the higher regularity for the second data in Theorem \ref{theorem4.1} in comparison with Theorem \ref{theorem3.1}. This comes from the treatment of the nonlinear function $f(u,u_t)= |u_t|^p$ appearing in (\ref{equation4.1}) in the place of the nonlinear function $f(u,u_t)= |u|^p$ appearing in (\ref{equation3.1}).
\end{remark}
This paper is organized as follows:
Section \ref{Sec.Linear} is devoted to the preparation of decay estimates for solutions to the corresponding linear equation \eqref{equation1.2}, which play an essential role to prove the global (in time) existence of small data solutions in Theorems \ref{theorem3.1} and \ref{theorem4.1}.
We will give the proofs of Theorem \ref{theorem3.1} and Theorem \ref{theorem4.1} in Section \ref{Semi-linear.1} and Section \ref{Semi-linear.2}, respectively.
Some further discussions including the large time behavior of the obtained global solutions and several results for \eqref{equation1.1} with the mixture of nonlinearities will be provided in Section \ref{Further.Sec}.
Finally, to end this paper, 
we summarize the well-known estimates, which are useful to show the decay properties of solutions in the appendix.
\section{The treatment of the linear equation} \label{Sec.Linear}
In this section, at first, we are going to prove the decay estimates for the fundamental solutions to \eqref{equation1.2}.
Some of them are already derived in the previous papers \cite{IkehataIyota,FukushimaIkehataMichihisa}.
However, for the convenience of the readers, we rephrase them to our notation.
\subsection{Representation of solutions}
At first, using partial Fourier transformation to (\ref{equation1.2}) we obtain the following Cauchy problem:
\begin{equation}
\widehat{u}_{tt}+ \nu |\xi|^4 \widehat{u}_t+ |\xi|^2 \widehat{u}=0,\quad \widehat{u}(0,\xi)= \widehat{u_0}(\xi),\quad \widehat{u}_t(0,\xi)= \widehat{u_1}(\xi). \label{equation2.1}
\end{equation}
The characteristic roots are
$$ \lambda_{\pm}=\lambda_{\pm}(\xi)= \frac{1}{2}\Big(-\nu|\xi|^4 \pm \sqrt{\nu^2 |\xi|^8- 4|\xi|^2}\Big). $$
The solutions to (\ref{equation2.1}) are presented by the following formula (here we assume $\lambda_{+}\neq \lambda_{-}$):
\begin{align*}
\widehat{u}(t,\xi) &= \frac{\lambda_+ e^{\lambda_- t}-\lambda_- e^{\lambda_+ t}}{\lambda_+- \lambda_-}\widehat{u_0}(\xi)+ \frac{e^{\lambda_+ t}-e^{\lambda_- t}}{\lambda_+- \lambda_-}\widehat{u_1}(\xi) \\ 
&=: \widehat{\mathcal{K}_0}(t,\xi)\widehat{u_0}(\xi)+\widehat{\mathcal{K}_1}(t,\xi)\widehat{u_1}(\xi),
\end{align*}
which leads to the representation formula of solutions to (\ref{equation1.2}) as follows:
\begin{equation} \label{eq:2.105}
u(t,x)= \mathcal{K}_0(t,x) \ast_x u_0(x)+ \mathcal{K}_1(t,x) \ast_x u_1(x). 
\end{equation}
Here $\mathcal{K}_k(t,x)$ with $k=0,1$ are the inverse Fourier transformation of $\widehat{\mathcal{K}_k}(t,\xi)$ with respect to spatial variables.

\subsection{Pointwise estimates in Fourier space}
Taking account of the cases of small and large frequencies separately we have the asymptotic behavior of the characteristic roots as follows:
\begin{align*}
&\lambda_{\pm} \sim -\frac{\nu}{2}|\xi|^4 \pm i|\xi|,\qquad \lambda_+-\lambda_- \sim 2i|\xi| \qquad \text{ for}\ |\xi| \ll 1 , \\
&\lambda_+\sim -\frac{1}{\nu}|\xi|^{-2},\qquad \lambda_- \sim - \nu |\xi|^4,\qquad \lambda_+- \lambda_- \sim \nu |\xi|^4 \qquad \text{ for}\ |\xi| \gg 1.
\end{align*}
To make further discussion rigorously, we follow the notation of \cite{IkehataTakeda} to introduce the radial, smooth cut-off functions $\chi_{L}$, $\chi_{M}$ and $\chi_{H}$ defined by 
\begin{gather*}
\chi_L (|\xi|) = \begin{cases}
	1 &\text{ if }\, |\xi| \leq \rho/2, \\
	0 &\text{ if }\, |\xi| \geq \rho, 
	\end{cases} \qquad
\chi_H (|\xi|) = \begin{cases}
	1 &\text{ if }\, |\xi| \geq 4\rho, \\
	0 &\text{ if }\, |\xi| \leq 2\rho, 
	\end{cases} \\ 
\chi_M (|\xi|) = 1- \chi_L (|\xi|) - \chi_H (|\xi|), 
\end{gather*}
where $\rho>0$ is chosen so that 
\begin{align*}
\rho< \frac{1}{2} \left( \frac{2}{\nu} \right)^{\frac{1}{3}}.
\end{align*}
As an easy consequence, we have the following point-wise estimates for the fundamental solutions in the Fourier space.
\begin{lemma} \label{lemma2.1}
Let $j=0,\,1$ and $s\ge 0$. Then, the following estimates hold:
\begin{align*}
|\xi|^s \chi_L(|\xi|)\big|\partial^j_t \widehat{\mathcal{K}_0}(t,\xi)\big| &\lesssim e^{-c t|\xi|^4} |\xi|^{s+j}, \\
|\xi|^s \chi_L(|\xi|)\big|\partial^j_t \widehat{\mathcal{K}_1}(t,\xi)\big| &\lesssim e^{-c t|\xi|^4} |\xi|^{s+j-1},
\end{align*}
and
\begin{align} \label{eq:2.2}
|\xi|^s \big(  \chi_M(|\xi|)+ \chi_H(|\xi|) \big)\big|\partial^j_t \widehat{\mathcal{K}_0}(t,\xi)\big| \lesssim
|\xi|^{s} \big(|\xi|^{-2j} e^{-ct|\xi|^{-2} } + |\xi|^{-6+4j} e^{-ct|\xi|^{4}}\big), \\
|\xi|^s   \big(  \chi_M(|\xi|)+ \chi_H(|\xi|) \big) \big|\partial^j_t \widehat{\mathcal{K}_1}(t,\xi)\big| \lesssim
|\xi|^{s-4} \big(|\xi|^{-2j} e^{-ct|\xi|^{-2} } + |\xi|^{4j} e^{-ct|\xi|^{4}}\big),
\label{eq:2.3}
\end{align}
where $c$ is a suitable positive constant.
\end{lemma}
\begin{proof}
Since the desired estimates for the low frequency part are shown in \cite{IkehataIyota}, we only show the estimates for the high frequency part, i.e. \eqref{eq:2.2} and \eqref{eq:2.3}. Observing that
\begin{equation*}
\begin{split}
\chi_H(|\xi|) \partial_{t}^{j} \widehat{\mathcal{K}_{0}}(t,\xi) & =  
\frac{\lambda_+ \lambda_{-}^{j} e^{\lambda_- t}-\lambda_- \lambda_{+}^{j} e^{\lambda_+ t}}{\lambda_+- \lambda_-}
\chi_H(|\xi|) \\
& \sim |\xi|^{-4} \big(|\xi|^{4-2j} e^{-ct|\xi|^{-2} } + |\xi|^{-2+4j} e^{-ct|\xi|^{4}}\big), 
\end{split}
\end{equation*}
we have the estimate \eqref{eq:2.2}. Here we notice that the characteristic roots $\lambda_{\pm}$ are negative in the middle frequency part $\big\{\xi\in \R^n \,:\, |\xi| \in \big[\rho/2,4\rho\big]\big\}$. So, the corresponding estimates for this part possess an exponential decay.
Similarly, we can obtain the estimate \eqref{eq:2.3}
by 
\begin{equation*}
\begin{split}
\chi_H(|\xi|) \partial_{t}^{j} \widehat{\mathcal{K}_{1}}(t,\xi) & =  
\frac{ \lambda_{+}^{j} e^{\lambda_+ t}-\lambda_{-}^{j} e^{\lambda_- t}}{\lambda_+- \lambda_-}
\chi_H(|\xi|) \\
& \sim |\xi|^{-4} \big(|\xi|^{-2j} e^{-ct|\xi|^{-2} } + |\xi|^{4j} e^{-ct|\xi|^{4}}\big).
\end{split}
\end{equation*}
This completes the proof of Lemma \ref{lemma2.1}.
\end{proof}
%
\subsection{Decay estimates}
Now we define the functions by 
\begin{equation}
\begin{split}
\mathcal{K}_{k \ell }(t,x):=\mathcal{F}^{-1}\big(
\widehat{\mathcal{K}_{k}}(t,\xi)
\chi_{\ell}(|\xi|) 
\big)
\end{split}
\end{equation}
for $k=0,1$ and $\ell=L,M,H$. 
Once we have Lemma \ref{lemma2.1}, we can easily conclude the decay properties of the frequency-wise evolution operators 
$\mathcal{K}_{k \ell }(t,x) \ast_x$ for $k=0,1$ and $\ell=L,M,H$ as follows. 
\begin{lemma}  \label{lemma2.2}
Let $\alpha\ge 0$, $j=0,\,1$ and $r \in [1,2]$. Then, the following estimates hold:
\begin{align}
\big\|\partial_t^j \nabla^\alpha \mathcal{K}_{0L}(t,x) \ast_x g \big\|_{L^2} &\lesssim (1+t)^{-\frac{n}{4}(\frac{1}{r}-\frac{1}{2})-\frac{\alpha+j}{4}}\|g\|_{L^r}, \label{lemma2.5.1} \\
\big\|\partial_t^j \nabla^\alpha \mathcal{K}_{1L}(t,x) \ast_x g \big\|_{L^2} &\lesssim
\begin{cases}
t^{\frac{1}{2}-\alpha} \|g\|_{L^1} &\text{ if }\, (n, j, \alpha) \in \{ 1 \} \times \{ 0 \} \times [0, \frac{1}{2}), \\
\sqrt{\log(t+e)} \|g\|_{L^1} &\text{ if }\, (n, j, \alpha) \in \big\{(1,0, \frac{1}{2}), (2,0,0)\big\}, \\
(1+t)^{-\frac{1}{4}(\frac{1}{r}-\frac{1}{2})-\frac{\alpha+j-1}{4}}\|g\|_{L^r} &\text{ if }\, (n, j, \alpha)\notin \{ 1 \} \times \{ 0 \} \times [0, \frac{1}{2}] \\
&\,\quad \text{ and }\, (n, j, \alpha)\neq (2,0,0),
\end{cases} \label{lemma2.5.2}
\end{align}
and 
\begin{align}
& \big\|\partial_t^j \nabla^\alpha \big( \mathcal{K}_{0M}(t,x) + \mathcal{K}_{0H}(t,x) \big) \ast_x g \big\|_{L^2} \nonumber \\
&\qquad \lesssim e^{-ct} t^{-\frac{n}{4}(\frac{1}{r}-\frac{1}{2}) -\frac{\alpha-6+4j-\beta_{1}}{4}}
  \big\| \nabla^{\beta_{1}} g \big\|_{L^r}
+ (1+t)^{-\frac{\beta_{2}}{2}} \big\| \nabla^{(\alpha-2j+\beta_{2})_{+}} g \big\|_{L^2}, \label{lemma2.5.3} \\
& \big\|\partial_t^j \nabla^\alpha \big( \mathcal{K}_{1M}(t,x)+\mathcal{K}_{1H}(t,x) \big) \ast_x g\big\|_{L^2} \nonumber \\
&\qquad \lesssim e^{-ct} t^{-\frac{n}{4}(\frac{1}{r}-\frac{1}{2}) -\frac{\alpha-4+4j-\beta_{1}}{4}}
  \big\| \nabla^{\beta_{1}} g \big\|_{L^r}
+ (1+t)^{-\frac{\beta_{2}}{2}} \big\| \nabla^{(\alpha-4-2j+\beta_{2})_{+}} g \big\|_{L^2}, \label{lemma2.5.4}
\end{align}
for any space dimensions $n\ge 1$ and $\beta_1,\beta_2\ge 0$, where $c$ is a suitable positive constant and $(\gamma)_{+}:= \max\{\gamma, 0 \}$ for $\gamma \in \R$.
\end{lemma}
\begin{proof}
One can find the proof of the estimates \eqref{lemma2.5.1} and \eqref{lemma2.5.2} in \cite{IkehataIyota,FukushimaIkehataMichihisa}
except for the case \eqref{lemma2.5.2} with $ (n, j, \alpha) \in \{ 1 \} \times \{ 0 \} \times (0, \frac{1}{2}]$.
On the other hand, applying the same arguments as in \cite{IkehataIyota,FukushimaIkehataMichihisa} we can easily obtain the following estimates for $n=1$:
\begin{equation*}
\begin{split}
\left\| 
\nabla^{\alpha} \mathcal{F}^{-1} 
\left(
e^{-\frac{\nu}{2} |\xi|^{4} t} \frac{\sin (t |\xi|)}{|\xi|}\chi_L (|\xi|)
\right) \ast_x g
\right\|_{L^{2}} \le 
\begin{cases}
t^{\frac{1}{2}-\alpha} \|g\|_{L^1} &\text{ if }\, \alpha \in [0, \frac{1}{2}), \\
\sqrt{\log(t+e)} \|g\|_{L^1} &\text{ if }\,  \alpha=\frac{1}{2},
\end{cases}
\end{split}
\end{equation*}
which are to conclude the remainder case. 
Then, it suffices to only prove the estimates \eqref{lemma2.5.3} and \eqref{lemma2.5.4}.
Firstly, let us prove the estimate \eqref{lemma2.5.3}.
Indeed, using the formula of Parseval-Plancherel and the estimate \eqref{eq:2.2} one derives
\begin{align*}
\big\|\partial_t^j \nabla^\alpha \big( \mathcal{K}_{0M}(t,x) + \mathcal{K}_{0H}(t,x) \big) \ast_x g \big\|_{L^2} &\lesssim \big\| |\xi|^{\alpha-2j} e^{-ct|\xi|^{-2}} \big(\chi_{M}(|\xi|)+\chi_{H}(|\xi|)\big) \widehat{g} \big\|_{L^2} \\
&\qquad + \big\| |\xi|^{\alpha-6+4j} e^{-ct|\xi|^{4}} \big(\chi_{M}(|\xi|)+\chi_{H}(|\xi|)\big) \widehat{g} \big\|_{L^2}.
\end{align*}
Now, we may see easily that 
\begin{equation*}
\begin{split}
\big\| |\xi|^{\alpha-2j} e^{-ct|\xi|^{-2}} \big(\chi_{M}(|\xi|)+\chi_{H}(|\xi|)\big) \widehat{g}  \big\|_{L^2} &=
\big\| |\xi|^{-\beta_{2}} e^{-ct|\xi|^{-2}}  \big(\chi_{M}(|\xi|)+\chi_{H}(|\xi|)\big) |\xi|^{\alpha-2j+\beta_{2}} \widehat{g} \big\|_{L^2} \\
&\le 
\begin{cases}
\big\| \nabla^{(\alpha-2j+\beta_{2})_{+}} g \big\|_{L^2} &\text{ if }\, 0 \le t \le 1, \\ 
C t^{-\frac{\beta_{2}}{2}} \big\| \nabla^{(\alpha-2j+\beta_{2})_{+}} g \big\|_{L^2} &\text{ if }\, t \ge 1, 
\end{cases}
\end{split}
\end{equation*}
where $C$ is a suitable positive constant, which gives
\begin{equation*}
\begin{split}
\big\| |\xi|^{\alpha-2j} e^{-ct|\xi|^{-2}} \big(\chi_{M}(|\xi|)+\chi_{H}(|\xi|)\big) \widehat{g}  \big\|_{L^2}
\lesssim (1+t)^{-\frac{\beta_{2}}{2}} \big\| \nabla^{(\alpha-2j+\beta_{2})_{+}} g \big\|_{L^2}.
\end{split}
\end{equation*}
In addition, we denote by $r'$, the conjugate number of $r$, i.e. $\frac{1}{r}+\frac{1}{r'}=1$. The application of H\"{o}lder's inequality and the Hausdorff-Young inequality entails
\begin{align*}
&\big\| |\xi|^{\alpha-6+4j} e^{-ct|\xi|^{4}} \big(\chi_{M}(|\xi|)+\chi_{H}(|\xi|)\big) \widehat{g} \big\|_{L^2} \\
&\qquad \lesssim \big\||\xi|^{\alpha-6+4j-\beta_1} e^{-ct|\xi|^{4}} \big(\chi_{M}(|\xi|)+\chi_{H}(|\xi|)\big)\big\|_{L^{\frac{2r}{2-r}}} \big\||\xi|^{\beta_1}\widehat{g}\big\|_{L^{r'}} \\
&\qquad \lesssim e^{-ct} t^{-\frac{n}{4}(\frac{1}{r}-\frac{1}{2}) -\frac{\alpha-6+4j-\beta_{1}}{4}}
  \big\| \nabla^{\beta_{1}} g \big\|_{L^r}.
\end{align*}
From the two above estimates, we can conclude the desired estimate \eqref{lemma2.5.3}. The remaining estimate is shown in a similar way. Hence, our proof is completed.
\end{proof}
From Lemma \ref{lemma2.2}, we may conclude immediately the following proposition.
\begin{proposition} \label{Prop.DecayEstimates}
Let $\ell_1,\,\ell_2 \ge 0$ and $r \in [1,2]$. Then, the solutions to \eqref{equation1.2} satisfy the following estimates:
\begin{align*}
\|u(t,\cdot)\|_{L^2}&\lesssim
\begin{cases}
(1+t)^{-\frac{1}{4}(\frac{1}{r}-\frac{1}{2})}\|u_0\|_{L^r}+ (1+t)^{-\frac{\ell_1}{2}}\|u_0\|_{\dot{H}^{\ell_1}} \\
\qquad + t^{\frac{1}{2}} \|u_1\|_{L^1}+ (1+t)^{-\frac{\ell_2}{2}-2}\|u_1\|_{\dot{H}^{\ell_2}} &\text{ if }\ n=1, \\
(1+t)^{-\frac{1}{2}(\frac{1}{r}-\frac{1}{2})}\|u_0\|_{L^r}+ (1+t)^{-\frac{\ell_1}{2}}\|u_0\|_{\dot{H}^{\ell_1}} \\
\qquad + \sqrt{\log(t+e)} \|u_1\|_{L^1}+ (1+t)^{-\frac{\ell_2}{2}-2}\|u_1\|_{\dot{H}^{\ell_2}} &\text{ if }\ n=2, \\
(1+t)^{-\frac{n}{4}(\frac{1}{r}-\frac{1}{2})}\|u_0\|_{L^r}+ (1+t)^{-\frac{\ell_1}{2}}\|u_0\|_{\dot{H}^{\ell_1}} \\
\qquad + (1+t)^{-\frac{n}{4}(\frac{1}{r}-\frac{1}{2})+\frac{1}{4}}\|u_1\|_{L^r}+ (1+t)^{-\frac{\ell_2}{2}-2}\| u_1\|_{\dot{H}^{\ell_2}} &\text{ if }\ n\ge 3,
\end{cases} \\ 
\|u(t,\cdot)\|_{\dot{H}^s}&\lesssim (1+t)^{-\frac{n}{4}(\frac{1}{r}-\frac{1}{2})-\frac{s}{4}}\|u_0\|_{L^r}+ (1+t)^{-\frac{\ell_1-s}{2}}\|u_0\|_{\dot{H}^{\ell_1}} \\
&\quad+ (1+t)^{-\frac{n}{4}(\frac{1}{r}-\frac{1}{2})-\frac{s-1}{4}}\|u_1\|_{L^r}+ (1+t)^{-\frac{\ell_2-s}{2}-2}\|u_1\|_{\dot{H}^{\ell_2}} \\
&\text{ if }\ 1\le s \le \min\{\ell_1+6,\ell_2+4\}, \\
\|u_t(t,\cdot)\|_{\dot{H}^s}&\lesssim (1+t)^{-\frac{n}{4}(\frac{1}{r}-\frac{1}{2})-\frac{s+1}{4}}\|u_0\|_{L^r}+ (1+t)^{-\frac{\ell_1-s}{2}-1}\|u_0\|_{\dot{H}^{\ell_1}} \\
&\quad+ (1+t)^{-\frac{n}{4}(\frac{1}{r}-\frac{1}{2})-\frac{s}{4}}\|u_1\|_{L^r}+ (1+t)^{-\frac{\ell_2-s}{2}-3}\|u_1\|_{\dot{H}^{\ell_2}} \\
&\text{ if }\ 0\le s \le \min\{\ell_1+2,\ell_2\},
\end{align*}
for any space dimensions $n\ge 1$.
\end{proposition}

\begin{remark}
\fontshape{n}
\selectfont
We want to point out that the obtained decay estimates for solutions to (\ref{equation1.2}) and several their derivatives in Proposition \ref{Prop.DecayEstimates} are regularity loss type estimates, which bring some difficulties to treat the associated semi-linear equations like (\ref{equation1.1}). However, to overcome such kind of difficulties, we can use appropriate regularities for the initial data by the suitable choice of $\ell_1$ and $\ell_2$ appearing in Proposition \ref{Prop.DecayEstimates}. Additionally, some technical steps of our proofs to deal with the nonlinear convection terms come into play in the next sections.
\end{remark}

\section{Proof of main result (I)} \label{Semi-linear.1}
In this section, we give the proof of Theorem \ref{theorem3.1}.
\begin{proof}[Proof of Theorem \ref{theorem3.1} with $n=2,3,4$]
We introduce the solution space
$$ X_{1}(t):= \mathcal{C}\big([0,\ity),H^{\frac{n}{2}+\e}\big)\cap \mathcal{C}^1\big([0,\ity),L^{2} \big) $$
with the norm
\begin{align*}
\|u\|_{X_{1}(t)}:= \sup_{0\le \tau \le t} \Big( \ell(\tau)\|u(\tau,\cdot)\|_{L^2} &+ (1+\tau)^{\frac{n-1+\e}{4}}\big\|\nabla^{\frac{n}{2}+\e} u(\tau,\cdot)\big\|_{L^2}+ (1+\tau)^{\frac{n}{8}}\|u_t(\tau,\cdot)\|_{L^2}\Big),
\end{align*}
where
$$ \ell(\tau)= \begin{cases}
\log^{-\frac{1}{2}}(\tau+e) &\text{ if }n=2, \\
(1+\tau)^{\frac{n}{8}-\frac{1}{4}} &\text{ if }n=3,4.
\end{cases} $$
As mentioned in \eqref{eq:2.105}, we can write the solutions of the corresponding linear Cauchy problem with vanishing right-hand side to (\ref{equation3.1}) as follows:
$$ u^{\text{ln}}(t,x)= \mathcal{K}_0(t,x) \ast_{x} u_0(x)+ \mathcal{K}_1(t,x) \ast_{x} u_1(x). $$
Using Duhamel's principle we get the formal implicit representation of solutions to (\ref{equation3.1}) in the following form:
$$ u(t,x)= u^{\text{ln}}(t,x) + \int_0^t \mathcal{K}_1(t-\tau,x) \ast_x \big(a \circ \nabla |u(\tau,x)|^p\big) d\tau=: u^{\text{ln}}(t,x)+ u^{\text{nl}}(t,x). $$
We define a mapping $\Phi_{1}: \,\, X_{1}(t) \longrightarrow X_{1}(t)$ by
$$ \Phi_{1}[u](t,x)= u^{\text{ln}}(t,x)+ u^{\text{nl}}(t,x). $$
In order to conclude the uniqueness and the global (in time) existence of small data solutions to (\ref{equation3.1}) as well, we have to prove the following pair of inequalities:
\begin{align}
\|\Phi_{1}[u]\|_{X_{1}(t)}& \lesssim \|(u_0,u_1)\|_{\mathcal{A}}+ \|u\|^p_{X_{1}(t)}, \label{inequality3.1} \\
\|\Phi_{1}[u]- \Phi_{1}[v]\|_{X_{1}(t)}& \lesssim \|u-v\|_{X_{1}(t)} \big(\|u\|^{p-1}_{X_{1}(t)}+ \|v\|^{p-1}_{X_{1}(t)}\big). \label{inequality3.2}
\end{align}
We firstly compute the norms of the nonlinear term in $X_{1}(t)$. 
Applying Proposition \ref{SobolevEmbedding}, from the definition of the norm in $X_{1}(t)$ we have the following useful estimate:
\begin{align*}
\big\|u(\tau,\cdot)\big\|_{L^{\infty}} \lesssim
\begin{cases}
\big(\log(\tau +e) \big)^{\frac{\varepsilon}{2(1+\varepsilon)}} (1+\tau)^{-\frac{1}{4}}
\|u\|_{X_{1}(t)} &\quad \text{if}\ n=2 , \\ 
(1+\tau)^{-\frac{n-1}{4}}
\|u\|_{X_{1}(t)} &\quad \text{if}\ n=3,4.
\end{cases}
\end{align*}
This leads to the following estimates for the nonlinear terms for $p\ge 2$:
\begin{align*}
\big\||u(\tau,\cdot)|^p\big\|_{L^r} &\le \|u(\tau,\cdot)\|^{p-(3-r)}_{L^{\infty}} 
\|u(\tau,\cdot)\|^{3-r}_{L^{2}} \\
&\lesssim 
\begin{cases}
(1+\tau)^{-\frac{1}{2}(p-\frac{1}{r})+\frac{p}{4}+\varepsilon_{1}}
\|u\|^p_{X_{1}(t)} &\quad \text{if}\ n=2, \\ 
(1+\tau)^{-\frac{n}{4}(p-\frac{1}{r})+\frac{p}{4}}\|u\|^p_{X_{1}(t)} &\quad \text{if}\ n=3,4,
\end{cases}
\end{align*}
with the choice of a sufficiently small constant $\varepsilon_{1}$ satisfying  $0< \varepsilon_{1}< \frac{p-5}{4}$ when $n=2$. Thus, it follows that
$$ \big\||u(\tau,\cdot)|^p\big\|_{L^r} \lesssim (1+\tau)^{-\frac{n}{4}(p-\frac{1}{r})+\frac{p}{4}+\epsilon_{0}}\|u\|^p_{X_{1}(t)} $$
with $n=2,3,4$ and $r=1,2$, where $\epsilon_{0}$ is given by
\begin{equation} \label{eq:3.3}
\begin{split}
\epsilon_{0}= 
\begin{cases}
\varepsilon_{1} \in \big(0,\frac{p-5}{4}\big)&\quad \text{if}\ n=2, \\
0 &\quad \text{if}\ n=3,4.
\end{cases} 
\end{split}
\end{equation}
In addition, one also derives
$$ \big\|\nabla |u(\tau,\cdot)|^p\big\|_{L^2} \le  \|u(\tau,\cdot)\|^{p-1}_{L^{\infty}} \| \nabla u(\tau,\cdot) \|_{L^2}. $$
The application of Proposition \ref{fractionalGagliardoNirenberg} gives
\begin{align*}
\| \nabla u(\tau,\cdot) \|_{L^2} &\lesssim \|u(\tau,\cdot) \|^{1-\theta}_{L^2}\,\|u(\tau,\cdot) \|^{\theta}_{\dot{H}^{\frac{n}{2}+\varepsilon}}\qquad \text{with }\theta= \frac{1}{\frac{n}{2}+\e} \\
&\lesssim \begin{cases}
\big(\log(\tau +e) \big)^{\frac{\varepsilon}{2(1+\varepsilon)}} (1+\tau)^{-\frac{1}{4}}
\|u\|_{X_{1}(t)} &\quad \text{if}\ n=2 , \\ 
(1+\tau)^{-\frac{n}{8}}
\|u\|_{X_{1}(t)} &\quad \text{if}\ n=3,4.
\end{cases}
\end{align*}
Thus, it follows immediately
\begin{align*}
\big\|\nabla |u(\tau,\cdot)|^p\big\|_{L^2} \le  \|u(\tau,\cdot)\|^{p-1}_{L^{\infty}} \| \nabla u(\tau) \|_{2} &
\lesssim (1+\tau)^{-\frac{n}{4}(p-\frac{1}{2})+\frac{p}{4}-\frac{1}{4}+\epsilon_0}\|u\|^p_{X_{1}(\tau)}.
\end{align*}
\par \textit{First let us prove the inequality \eqref{inequality3.1}.} From the definition of the data space, it is obvious that we need to indicate the following inequality instead of \eqref{inequality3.1}:
\begin{equation}
\|u^{\text{nl}}\|_{X_{1}(t)} \lesssim \|u\|^p_{X_{1}(t)}. \label{inequality3.3}
\end{equation}
Our proof is divided into two steps. \medskip

Step 1: $\quad$ We may control the norm $\|u^{\text{nl}}(t,\cdot)\|_{L^2}$ by using the estimates from Lemma \ref{lemma2.2} as follows:
\begin{align*}
\|u^{\text{nl}}(t,\cdot)\|_{L^2} &\lesssim \int_0^t \big\|\nabla \mathcal{K}_{1L}(t-\tau,x) \ast_x |u(\tau,x)|^p\big\|_{L^2} d\tau \\
&\qquad+ \int_0^t \big\|\nabla \big( \mathcal{K}_{1M}(t-\tau,x) + \mathcal{K}_{1H}(t-\tau,x) \big) \ast_x |u(\tau,x)|^p\big\|_{L^2} d\tau \\
&\lesssim \int_0^{t/2} (1+t-\tau)^{-\frac{n}{8}}\big\||u(\tau,\cdot)|^p\big\|_{L^1} d\tau+ \int_{t/2}^t \big\||u(\tau,\cdot)|^p\big\|_{L^2} d\tau \\
&\qquad + \int_0^t (1+t-\tau)^{-\frac{3}{2}}\big\||u(\tau,\cdot)|^p\big\|_{L^2} d\tau \\
&\lesssim \Big(\int_0^{t/2} (1+t-\tau)^{-\frac{n}{8}}(1+\tau)^{-\frac{n}{4}(p-1)+\frac{p}{4}+\epsilon_{0}} d\tau+ \int_{t/2}^t (1+\tau)^{-\frac{n}{4}(p-\frac{1}{2})+\frac{p}{4}+\epsilon_{0}} d\tau \\
&\qquad +\int_0^t (1+t-\tau)^{-\frac{3}{2}}(1+\tau)^{-\frac{n}{4}(p-\frac{1}{2})+\frac{p}{4}+\epsilon_{0}} d\tau \Big) \|u\|^p_{X_{1}(t)}.
\end{align*}
Noting the assumptions \eqref{exponent3.1} and \eqref{eq:3.3}, 
we easily see $-\frac{n}{4}(p-1)+\frac{p}{4}+\epsilon_{0} \le -\frac{3}{4}$.
The direct calculation shows that the first two integrals are estimated as follows: 
\begin{align*}
\int_0^{t/2} (1+t-\tau)^{-\frac{n}{8}}(1+\tau)^{-\frac{n}{4}(p-1)+\frac{p}{4}+\epsilon_{0}} d\tau &\lesssim 
\begin{cases} 
(1+t)^{-\frac{n}{8}}\log(t+e) &\text{ if } \ p \ge 1+\frac{5+4\epsilon_0}{n-1}, \\  
(1+t)^{-\frac{n}{8}-\frac{n}{4}(p-1)+\frac{p}{4}+\epsilon_{0}+1} &\text{ if } \ p < 1+\frac{5+4\epsilon_0}{n-1},  
\end{cases} \\
&\lesssim (1+t)^{-\frac{n}{8}+\frac{1}{4}}
\end{align*}
with $p$ satisfying \eqref{exponent3.1},
and
$$\int_{t/2}^t (1+\tau)^{-\frac{n}{4}(p-\frac{1}{2})+\frac{p}{4}+\epsilon_{0}} d\tau 
\lesssim (1+t)^{-\frac{n}{4}(p-\frac{1}{2})+\frac{p}{4}+1+\epsilon_{0}}
\lesssim (1+t)^{-\frac{n}{8}+\frac{1}{4}}. $$
After applying Lemma \ref{LemmaIntegral.1}, we can obtain the estimate for the third integral:
\begin{align*}
\int_0^t (1+t-\tau)^{-\frac{3}{2}}(1+\tau)^{-\frac{n}{4}(p-\frac{1}{2})+\frac{p}{4}+\epsilon_{0}} d\tau &\lesssim (1+\tau)^{-\min\big\{\frac{3}{2},\,\frac{n}{4}(p-\frac{1}{2})-\frac{p}{4}-\epsilon_{0}\big\}} \\ 
&\lesssim (1+t)^{-\frac{n}{8}-\frac{3}{4}} \le (1+t)^{-\frac{n}{8}+\frac{1}{4}}
\end{align*}
since 
$-\frac{3}{2}< -\frac{n}{8}-\frac{3}{4}$ for $n \le 4$ and $-\frac{n}{4}(p-\frac{1}{2})+ \frac{p}{4} +\epsilon_{0} \le -\frac{n}{8}-\frac{3}{4}$ under the assumptions \eqref{exponent3.1} and \eqref{eq:3.3}.
Therefore, we can conclude that
$$ \|u^{\text{nl}}(t,\cdot)\|_{L^2} \lesssim
(1+t)^{-\frac{n}{8}+ \frac{1}{4}} \|u\|^p_{X_1(t)} \ \text{ for }n=2,3,4. $$
Step 2: $\quad$ By using the same ideas, we may deal with the remaining norms
$$ \big\|\nabla^{\frac{n}{2}+\e} u^{\text{nl}}(t,\cdot)\big\|_{L^2} \quad \text{ and }\quad \|u_t^{\text{nl}}(t,\cdot)\|_{L^2} $$
as follows:
\begin{align*}
\big\|\nabla^{\frac{n}{2}+\e} u^{\text{nl}}(t,\cdot)\big\|_{L^2} &\lesssim 
\int_0^t \big\|\nabla^{\frac{n}{2}+1+\e} \mathcal{K}_{1L}(t-\tau,x) \ast_x |u(\tau,x)|^p\big\|_{L^2} d\tau \\
&\qquad+ \int_0^t \big\|\nabla^{\frac{n}{2}+\e} \big( \mathcal{K}_{1M}(t-\tau,x) + \mathcal{K}_{1H}(t-\tau,x) \big)\ast_x \nabla |u(\tau,x)|^p\big\|_{L^2} d\tau \\
&\lesssim \int_0^{t/2} (1+t-\tau)^{-\frac{n+\e}{4}}\big\||u(\tau,\cdot)|^p\big\|_{L^1} d\tau+ \int_{t/2}^t (1+t-\tau)^{-\frac{1}{4}(\frac{n}{2}+\e)}\big\||u(\tau,\cdot)|^p\big\|_{L^2} d\tau \\
&\qquad+ \int_0^t (1+t-\tau)^{-2+\frac{n}{4}+\frac{\e}{2}}\big\|\nabla |u(\tau,\cdot)|^p\big\|_{L^2} d\tau \\
&\lesssim \Big(\int_0^{t/2} (1+t-\tau)^{-\frac{n+\e}{4}}(1+\tau)^{-\frac{n}{4}(p-1)+\frac{p}{4}+\epsilon_{0}} d\tau \\
&\qquad+ \int_{t/2}^t (1+t-\tau)^{-\frac{1}{4}(\frac{n}{2}+\e)}(1+\tau)^{-\frac{n}{4}(p-\frac{1}{2})+\frac{p}{4}+\epsilon_{0}} d\tau \\
&\qquad+\int_0^t (1+t-\tau)^{-2+\frac{n}{4}+\frac{\e}{2}}(1+\tau)^{-\frac{n}{4}(p-\frac{1}{2})+\frac{p-1}{4}} d\tau \Big)\|u\|^p_{X_{1}(t)}.
\end{align*}
Then, repeating some arguments as we did in Step 1 we may conclude
$$ \big\|\nabla^{\frac{n}{2}+\e} u^{\text{nl}}(t,\cdot)\big\|_{L^2} \lesssim (1+t)^{-\frac{n-1+\e}{4}} \|u\|^p_{X_{1}(t)}. $$
By an analogous manner, we can proceed as follows:
\begin{align*}
\|u_t^{\text{nl}}(t,\cdot)\|_{L^2} &\lesssim \int_0^t \big\|\partial_t\nabla \mathcal{K}_{1L}(t-\tau,x) \ast_x |u(\tau,x)|^p\big\|_{L^2} d\tau \\
&\qquad+ \int_0^t \big\|\partial_t\nabla \big( \mathcal{K}_{1M}(t-\tau,x) + \mathcal{K}_{1H}(t-\tau,x) \big)\ast_x |u(\tau,x)|^p\big\|_{L^2} d\tau \\
&\lesssim \int_0^{t/2} (1+t-\tau)^{-\frac{n}{8}-\frac{1}{4}}\big\||u(\tau,\cdot)|^p\big\|_{L^1} d\tau+ \int_{t/2}^t (1+t-\tau)^{-\frac{1}{4}}\big\||u(\tau,\cdot)|^p\big\|_{L^2} d\tau \\
&\qquad+ \int_0^t e^{-c(t-\tau)} \big\|\nabla |u(\tau,\cdot)|^p\big\|_{L^2} d\tau+ \int_0^t (1+t-\tau)^{-\frac{5}{2}} \big\||u(\tau,\cdot)|^p\big\|_{L^2} d\tau
\end{align*}
\begin{align*}
&\lesssim \Big(\int_0^{t/2} (1+t-\tau)^{-\frac{n}{8}-\frac{1}{4}}(1+\tau)^{-\frac{n}{4}(p-1)+\frac{p}{4}+\epsilon_{0}} d\tau \\
&\qquad + \int_{t/2}^t (1+t-\tau)^{-\frac{1}{4}}(1+\tau)^{-\frac{n}{4}(p-\frac{1}{2})+\frac{p}{4}+\epsilon_{0}} d\tau \\
&\qquad +\int_0^t e^{-c(t-\tau)} (1+\tau)^{-\frac{n}{4}(p-\frac{1}{2})+\frac{p}{4}- \frac{1}{4}+\epsilon_{0}} d\tau \\
&\qquad +\int_0^t  (1+t-\tau)^{-\frac{5}{2}} (1+\tau)^{-\frac{n}{4}(p-\frac{1}{2})+\frac{p}{4}+\epsilon_{0}} d\tau \Big)\|u\|^p_{X_{1}(t)}.
\end{align*}
For this reason, repeating again some arguments as we did in Step 1 one also obtains
\begin{align*}
\int_0^{t/2} (1+t-\tau)^{-\frac{n}{8}-\frac{1}{4}}(1+\tau)^{-\frac{n}{4}(p-1)+\frac{p}{4}+\epsilon_{0}} d\tau &\lesssim (1+t)^{-\frac{n}{8}} \\
\int_{t/2}^t (1+t-\tau)^{-\frac{1}{4}}(1+\tau)^{-\frac{n}{4}(p-\frac{1}{2})+\frac{p}{4}+\epsilon_{0}} d\tau &\lesssim (1+t)^{-\frac{n}{8}} \\
\int_0^t  (1+t-\tau)^{-\frac{5}{2}} (1+\tau)^{-\frac{n}{4}(p-\frac{1}{2})+\frac{p}{4}+\epsilon_{0}} d\tau &\lesssim (1+t)^{-\frac{n}{8}}.
\end{align*}
To estimate the remaining integral, we shall employ Lemma \ref{LemmaIntegral.2} to achieve
\begin{align*}
\int_0^t e^{-c(t-\tau)} (1+\tau)^{-\frac{n}{4}(p-\frac{1}{2})+\frac{p}{4}- \frac{1}{4}+\epsilon_{0}} d\tau &\lesssim (1+t)^{-\frac{n}{4}(p-\frac{1}{2})+\frac{p}{4}- \frac{1}{4}+\epsilon_{0}} \\ 
&\lesssim (1+t)^{-\frac{n}{8}-1} \le (1+t)^{-\frac{n}{8}},
\end{align*}
where we have used $-\frac{n}{4}(p-\frac{1}{2})+ \frac{p}{4}- \frac{1}{4} +\epsilon_{0} \le -\frac{n}{8}-1$ due to the assumptions \eqref{exponent3.1} and \eqref{eq:3.3}. 
Thus, we arrive at the following estimates:
$$ \|u_t^{\text{nl}}(t,\cdot)\|_{L^2} \lesssim (1+t)^{-\frac{n}{8}} \|u\|^p_{X_{1}(t)}. $$
Therefore, from the definition of the norm in $X_{1}(t)$ we obtain immediately the inequality (\ref{inequality3.3}).
\par \textit{Next let us prove the inequality \eqref{inequality3.2}.} 
We shall follow the strategy used in the proof of the inequality (\ref{inequality3.3}). The new difficulty is to require the estimates for the term
$$|u(\tau,\cdot)|^p- |v(\tau,\cdot)|^p $$
in $L^1$, $L^2$ and $\dot{H}^{1}$. Then, repeating an analogous treatment as in the proof of the inequality (\ref{inequality3.3}) we may conlcude the inequality (\ref{inequality3.2}). Indeed, by using H\"{o}lder's inequality we get
\begin{align*}
\big\||u(\tau,\cdot)|^p-|v(\tau,\cdot)|^p\big\|_{L^1} &\lesssim \|u(\tau,\cdot)- v(\tau,\cdot)\|_{L^{p}} \big(\|u(\tau,\cdot)\|^{p-1}_{L^{p}}+\|v(\tau,\cdot)\|^{p-1}_{L^{p}}\big), \\ 
\big\||u(\tau,\cdot)|^p- |v(\tau,\cdot)|^p\big\|_{L^2} &\lesssim \|u(\tau,\cdot)- v(\tau,\cdot)\|_{L^{2p}} \big(\|u(\tau,\cdot)\|^{p-1}_{L^{2p}}+\|v(\tau,\cdot)\|^{p-1}_{L^{2p}}\big).
\end{align*}
Analogously to the proof of (\ref{inequality3.3}), employing Proposition \ref{fractionalGagliardoNirenberg} to the norms
$$ \|u(\tau,\cdot)- v(\tau,\cdot)\|_{L^\eta}, \quad \|u(\tau,\cdot)\|_{L^\eta}, \quad \|v(\tau,\cdot)\|_{L^\eta} $$
with $\eta=p$ and $\eta=2p$ we may arrive at the following estimates:
\begin{align*}
\big\||u(\tau,\cdot)|^p-|v(\tau,\cdot)|^p\big\|_{L^1} &\lesssim (1+\tau)^{-\frac{n}{4}(p-1)+\frac{p}{4}+\epsilon_{0}} \|u-v\|_{X_{1}(t)}\big(\|u\|^{p-1}_{X_{1}(t)}+\|v\|^{p-1}_{X_{1}(t)}\big),\\
\big\||u(\tau,\cdot)|^p- |v(\tau,\cdot)|^p\big\|_{L^2} &\lesssim (1+\tau)^{-\frac{n}{4}(p-\frac{1}{2})+\frac{p}{4}+\epsilon_{0}} \|u-v\|_{X_{1}(t)}\big(\|u\|^{p-1}_{X_{1}(t)}+\|v\|^{p-1}_{X_{1}(t)}\big).
\end{align*}
Let us now turn to estimate the norm
$$\big\||u(\tau,\cdot)|^p-|v(\tau,\cdot)|^p\big\|_{\dot{H}^{1}}= \big\|\nabla \big(|u(\tau,\cdot)|^p-|v(\tau,\cdot)|^p\big)\big\|_{L^{2}}. $$
At first, observing that 
\begin{equation*}
\frac{\text{d}}{\text{d}u} |u|^{p-2} u=(p-1)|u|^{p-2}
\end{equation*}
for $p \ge 2$ 
we apply the mean value theorem to have 
\begin{equation*}
\begin{split}
|u|^{p-2} u -|v|^{p-2} v = (p-1) \big(\omega |u|^{p-2} +(1-\omega) |v|^{p-2}\big)(u-v) 
\end{split}
\end{equation*}
for some $\omega \in [0,1]$.
On the  other hand, we see
\begin{equation*}
\begin{split}
&\nabla \big(|u(\tau,x)|^p-|v(\tau,x)|^p\big) \\
&\qquad= p |u(\tau,x)|^{p-2} u(\tau,x) \nabla u(\tau,x)-p |v(\tau,x)|^{p-2} v(\tau,x) \nabla v(\tau,x) \\
&\qquad = p |u(\tau,x)|^{p-2} u(\tau,x) \big(\nabla u(\tau,x) - \nabla v(\tau,x)\big) \\
&\qquad\quad + p \nabla v(\tau,x) \big(|u(\tau,x)|^{p-2} u(\tau,x) -|v(\tau,x)|^{p-2} v(\tau,x)\big).
\end{split}
\end{equation*}
Therefore, we have 
\begin{align*}
&\big|\nabla \big(|u(\tau,x)|^p-|v(\tau,x)|^p\big)\big| \\
&\qquad \le C |u(\tau,x)|^{p-1}  |\nabla u(\tau,x) - \nabla v(\tau,x)| \\ 
&\qquad\quad + C  |\nabla v(\tau,x)| \big(|u(\tau,x)|^{p-2}+|v(\tau,x)|^{p-2}\big) |u(\tau,x)-v(\tau,x)|.
\end{align*}
Then, we can conclude the estimate
\begin{equation*}
\begin{split}
&\big\| \nabla \big(|u(\tau,\cdot)|^p-|v(\tau,\cdot)|^p\big) \big\|_{L^2} \\
&\qquad \le C \|u(\tau, \cdot) \|_{L^\infty}^{p-1}  \|\nabla u(\tau, \cdot)  - \nabla v(\tau, \cdot) \|_{L^2} \\
&\qquad \quad+ C  \|\nabla v(\tau, \cdot) \|_{L^2} \big(\|u(\tau, \cdot) \|_{L^\infty}^{p-2} 
+\|v(\tau, \cdot) \|_{L^\infty}^{p-2}\big) \|u(\tau, \cdot) -v(\tau, \cdot) \|_{L^\infty},
\end{split}
\end{equation*}
which implies the desired estimate 
\begin{align}
&\big\|\nabla \big(|u(\tau,\cdot)|^p-|v(\tau,\cdot)|^p\big) \big\|_{L^2} \nonumber \\
&\qquad \lesssim
(1+\tau)^{-\frac{n}{4}(p-\frac{1}{2})+\frac{p}{4}-\frac{1}{4}} \|u-v\|_{X_{1}(t)}\big(\|u\|^{p-1}_{X_{1}(t)}+\|v\|^{p-1}_{X_{1}(t)}\big) \label{ineq:3.4}
\end{align}
by the aid of Proposition \ref{fractionalGagliardoNirenberg}.
This completes the proof of inequality (\ref{inequality3.2}).
\end{proof}
\begin{proof}[Proof of Theorem \ref{theorem3.1} with $n=5$]
We need to modify the solution space as
$$ X_{2}(t):= \mathcal{C}\big([0,\ity), H^{\frac{5}{2}+\e}\big)\cap \mathcal{C}^1\big([0,\ity),L^{2} \big) $$
with the norm
\begin{align*}
\|u\|_{X_{2}(t)}:= \sup_{0\le \tau \le t} \Big( (1+\tau)^{\frac{3}{8}} \|u(\tau,\cdot)\|_{L^2} &+ (1+\tau)^{\frac{3}{4}-\frac{\e}{2}}\big\|\nabla^{\frac{5}{2}+\e} u(\tau,\cdot)\big\|_{L^2}+ (1+\tau)^{\frac{5}{8}}\|u_t(\tau,\cdot)\|_{L^2}\Big).
\end{align*}
At first, we note that
\begin{equation*}
\| u(\tau,\cdot) \|_{\infty} \lesssim \| u(\tau,\cdot)\|_{2}^{1-\theta_{0}} \big\| \nabla^{\frac{5}{2}+\e} u(\tau,\cdot) \big\|_{2}^{\theta_{0}} \lesssim (1+\tau)^{\frac{-15+7\e}{4(5+2\e)}}\|u\|_{X_{2}(t)} 
\end{equation*}
with $\theta_{0}=\frac{\frac{5}{2}}{\frac{5}{2}+\e}$ by Proposition \ref{SobolevEmbedding}, and 
\begin{equation*}
\| \nabla u(\tau,\cdot)\|_{2} \lesssim \| u(\tau,\cdot) \|_{2}^{1-\theta_{1}} \big\| \nabla^{\frac{5}{2}+\e} u(\tau,\cdot) \big\|_{2}^{\theta_{1}} \lesssim (1+\tau)^{\frac{-21+2\e}{8(5+2\e)}}\|u\|_{X_{2}(t)}
\end{equation*}
with $\theta_{1}= \frac{1}{\frac{5}{2}+\e}$ by Proposition \ref{fractionalGagliardoNirenberg}.
Now we choose a constant $\e_2$ fulfilling $\varepsilon_{2}\ge \frac{13 \e}{4(5+2 \varepsilon)}$, which solves 
$\frac{-15+7 \e}{4(5+2 \e)} \le -\frac{3}{4} +\e_{2}$ and 
$\frac{-21+2 \e}{8(5+2 \e)} \le -\frac{21}{40}+\e_{2}$, 
so that one arrives at the following estimates:
\begin{align*}
\| u(\tau,\cdot) \|_{\infty} &\lesssim (1+\tau)^{-\frac{3}{4}+\varepsilon_{2} }\|u\|_{X_{2}(t)}, \\ 
\| \nabla u(\tau,\cdot)\|_{2} &\lesssim (1+\tau)^{-\frac{21}{40}+\varepsilon_{2}}\|u\|_{X_{2}(t)}.
\end{align*}
Then we define a mapping $\Phi_{2}: \,\, X_{2}(t) \longrightarrow X_{2}(t)$ by
$$ \Phi_{2}[u](t,x)= u^{\text{ln}}(t,x)+ u^{\text{nl}}(t,x). $$
As we see in the proof of Theorem \ref{theorem3.1} with $n=2,3,4$, the proof of Theorem \ref{theorem3.1} with $n=5$ is reduced to prove the following estimates:
\begin{align}
\|\Phi_{2}[u]\|_{X_{2}(t)}& \lesssim \|(u_0,u_1)\|_{\mathcal{A}}+ \|u\|^p_{X_{2}(t)}, \label{inequality5.1} \\
\|\Phi_{2}[u]- \Phi_{2}[v]\|_{X_{2}(t)}& \lesssim \|u-v\|_{X_{2}(t)} \big(\|u\|^{p-1}_{X_{2}(t)}+ \|v\|^{p-1}_{X_{2}(t)}\big). \label{inequality5.2}
\end{align}
Similar arguments to the proof of Theorem \ref{theorem3.1} with $n=2,3,4$ yield the estimates for the nonlinear terms as follows:
\begin{align}
\big\||u(\tau,\cdot)|^p\big\|_{L^1} &\le \|u(\tau,\cdot)\|^{p-2}_{L^{\infty}} \|u(\tau,\cdot)\|^{2}_{L^{2}} 
\lesssim (1+\tau)^{-\frac{3}{4}(p-1)+\e_{2} (p-2)}\|u\|^p_{X_{2}(t)}, \label{eq:3.10} \\ 
\big\||u(\tau,\cdot)|^p\big\|_{L^2} &\le \|u(\tau,\cdot)\|^{p-1}_{L^{\infty}} \|u(\tau,\cdot)\|_{L^{2}}  
\lesssim (1+\tau)^{-\frac{3}{4}(p-\frac{1}{2})+\e_{2}(p-1)} \|u\|^p_{X_{2}(t)}, \label{eq:3.11} \\
\big\|\nabla |u(\tau,\cdot)|^p\big\|_{L^2} &\le \|u(\tau,\cdot)\|^{p-1}_{L^{\infty}} \|\nabla u(\tau,\cdot)\|_{L^{2}}  
\lesssim (1+\tau)^{-\frac{3}{4}(p-\frac{3}{10})+\e_{2} p} \|u\|^p_{X_{2}(t)}. \label{eq:3.12}
\end{align}
For the proof of Theorem \ref{theorem3.1} with $n=5$, the following form of the estimates from \eqref{eq:3.10} to \eqref{eq:3.12} are useful:
\begin{align*}
\big\||u(\tau,\cdot)|^p\big\|_{L^1} &\lesssim (1+\tau)^{-\frac{3}{4}}\|u\|^p_{X_{2}(t)},  \\ 
\big\||u(\tau,\cdot)|^p\big\|_{L^2} &\lesssim (1+\tau)^{-\frac{9}{8}+\e_{2}p} \|u\|^p_{X_{2}(t)},  \\
\big\|\nabla |u(\tau,\cdot)|^p\big\|_{L^2} &\lesssim (1+\tau)^{-\frac{51}{40}+\e_{2} p} \|u\|^p_{X_{2}(t)}. 
\end{align*}
They follow by the same method as in the previous section.
\par \textit{First let us prove the inequality \eqref{inequality5.1}.} As in the proof of the estimate \eqref{inequality3.1}, 
we only show the estimate
\begin{equation}
\|u^{\text{nl}}\|_{X_{2}(t)} \lesssim \|u\|^p_{X_{2}(t)}. \label{inequality5.3}
\end{equation}
Our proof is divided into two steps. \medskip

Step 1: $\quad$ We may estimate the norm $\|u^{\text{nl}}(t,\cdot)\|_{L^2}$ as follows:
\begin{align*}
\|u^{\text{nl}}(t,\cdot)\|_{L^2} &\lesssim \int_0^t \big\|\nabla \mathcal{K}_{1L}(t-\tau,x) \ast_x |u(\tau,x)|^p\big\|_{L^2} d\tau \\
&\qquad + \int_0^t \big\|\nabla \big( \mathcal{K}_{1M}(t-\tau,x) + \mathcal{K}_{1H}(t-\tau,x) \big) \ast_x |u(\tau,x)|^p\big\|_{L^2} d\tau \\
&\lesssim \int_0^{t} (1+t-\tau)^{-\frac{5}{8}}\big\||u(\tau,\cdot)|^p\big\|_{L^1} d\tau+ \int_0^t (1+t-\tau)^{-\frac{3}{2}}\big\||u(\tau,\cdot)|^p\big\|_{L^2} d\tau \\
&\lesssim \Big(\int_0^{t} (1+t-\tau)^{-\frac{5}{8}}(1+\tau)^{-\frac{3}{4}} d\tau+ \int_0^t (1+t-\tau)^{-\frac{3}{2}}(1+\tau)^{-\frac{9}{8}+\e_{2}p} d\tau \Big) \|u\|^p_{X_{2}(t)}.
\end{align*}
The employment of Lemma \ref{LemmaIntegral.1} implies immediately that
$$ \|u^{\text{nl}}(t,\cdot)\|_{L^2} \lesssim (1+t)^{-\frac{3}{8}}\|u\|^p_{X_{2}(t)}. $$
Step 2: $\quad$ By using the same ideas, we may control the remaining norms
$$ \big\|\nabla^{\frac{n}{2}+\e} u^{\text{nl}}(\tau,\cdot)\big\|_{L^2} \quad \text{ and }\quad \|u_t^{\text{nl}}(t,\cdot)\|_{L^2} $$
as follows:
\begin{align*}
\big\|\nabla^{\frac{5}{2}+\e} u^{\text{nl}}(t,\cdot)\big\|_{L^2} &\lesssim \int_0^t \big\|\nabla^{\frac{5}{2}+1+\e} \mathcal{K}_{1L}(t-\tau,x) \ast_x |u(\tau,x)|^p\big\|_{L^2} d\tau \\
&\qquad+ \int_0^t \big\|\nabla^{\frac{5}{2}+\e} \big( \mathcal{K}_{1M}(t-\tau,x) + \mathcal{K}_{1H}(t-\tau,x) \big)\ast_x \nabla |u(\tau,x)|^p\big\|_{L^2} d\tau \\
&\lesssim \int_0^{t} (1+t-\tau)^{-\frac{5+\e}{4}}\big\||u(\tau,\cdot)|^p\big\|_{L^1} d\tau+ \int_0^t (1+t-\tau)^{-2+\frac{5}{4}+\frac{\e}{2}}\big\|\nabla |u(\tau,\cdot)|^p\big\|_{L^2} d\tau \\
&\lesssim \Big(\int_0^{t} (1+t-\tau)^{-\frac{5+\e}{4}}(1+\tau)^{-\frac{3}{4}} d\tau \\
&\qquad+ \int_0^t (1+t-\tau)^{-\frac{3}{4}+\frac{\e}{2}}(1+\tau)^{-\frac{51}{40}+\e_{2}p} d\tau \Big)\|u\|^p_{X_{2}(t)}.
\end{align*}
Then, after applying Lemma \ref{LemmaIntegral.1} again, we may conclude the following estimate:
$$ \big\|\nabla^{\frac{5}{2}+\e} u^{\text{nl}}(t,\cdot)\big\|_{L^2} \lesssim (1+t)^{-\frac{3}{4}+\frac{\e}{2}} \|u\|^p_{X_{2}(t)}. $$
Now we turn to deal with the norm $\|u_t^{\text{nl}}(t,\cdot)\|_{L^2}$ by
\begin{align*}
\|u_t^{\text{nl}}(t,\cdot)\|_{L^2} &\lesssim \int_0^t \big\|\partial_t\nabla \mathcal{K}_{1L}(t-\tau,x) \ast_x |u(\tau,x)|^p\big\|_{L^2} d\tau \\
&\qquad+ \int_0^t \big\|\partial_t \nabla \big( \mathcal{K}_{1M}(t-\tau,x) + \mathcal{K}_{1H}(t-\tau,x) \big)
 \ast_x  |u(\tau,x)|^p\big\|_{L^2} d\tau \\
&\lesssim \int_0^{t} (1+t-\tau)^{-\frac{7}{8}}\big\||u(\tau,\cdot)|^p\big\|_{L^1} d\tau \\
&\qquad+ \int_0^t e^{-c(t-\tau)} \big\|\nabla |u(\tau,\cdot)|^p\big\|_{L^2} d\tau+ \int_0^t (1+t-\tau)^{-\frac{5}{2}} \big\||u(\tau,\cdot)|^p\big\|_{L^2} d\tau \\
&\lesssim \Big(\int_0^{t} (1+t-\tau)^{-\frac{7}{8}}(1+\tau)^{-\frac{3}{4}} d\tau + \int_0^t e^{-c(t-\tau)}(1+\tau)^{-\frac{51}{40}+\e_{2}p} d\tau \\
&\qquad+ \int_0^t (1+t-\tau)^{-\frac{5}{2}} (1+\tau)^{-\frac{9}{8}+\e_{2}p} d\tau \Big)\|u\|^p_{X_{2}(t)}.
\end{align*}
Then, using some arguments as we did in Step 2 in the proof of the case $n=2,3,4$ we arrive at the following estimate:
\begin{align*}
\|u_t^{\text{nl}}(t,\cdot)\|_{L^2} &\lesssim (1+t)^{-\frac{5}{8}} \|u\|^p_{X_{2}(t)}.
\end{align*}
Therefore, from the definition of the norm in $X_{2}(t)$ we obtain immediately the inequality (\ref{inequality5.3}).
\par \textit{Next let us prove the inequality \eqref{inequality5.2}.} 
By the same way as in the proof of the estimate \eqref{inequality3.2}, we have 
\begin{align*}
\big\||u(\tau,\cdot)|^p-|v(\tau,\cdot)|^p\big\|_{L^1} &\lesssim (1+\tau)^{-\frac{3}{4}} \|u-v\|_{X_{2}(t)}\big(\|u\|^{p-1}_{X_{2}(t)}+\|v\|^{p-1}_{X_{2}(t)}\big),\\
\big\||u(\tau,\cdot)|^p- |v(\tau,\cdot)|^p\big\|_{L^2} &\lesssim (1+\tau)^{-\frac{9}{8}+\e_{2}p} \|u-v\|_{X_{2}(t)}\big(\|u\|^{p-1}_{X_{2}(t)}+\|v\|^{p-1}_{X_{2}(t)}\big), \\
\big\| \nabla \big(|u(\tau,\cdot)|^p-|v(\tau,\cdot)|^p\big) \big\|_{L^2} &\lesssim
(1+\tau)^{-\frac{51}{40}+ \e_{2}p} \|u-v\|_{X_{2}(t)}\big(\|u\|^{p-1}_{X_{2}(t)}+\|v\|^{p-1}_{X_{2}(t)}\big).
\end{align*}
This completes the proof of inequality (\ref{inequality5.2}).
\end{proof}

\section{Proof of main result (II)} \label{Semi-linear.2}
This section is devoted to the proof of Theorem \ref{theorem4.1}. 
\begin{proof}[Proof of Theorem \ref{theorem4.1}]
We introduce the solution space
$$ Y(t):= \mathcal{C}\big([0,t],H^{\frac{n}{2}+\e}\big)\cap \mathcal{C}^1\big([0,t],H^{\frac{n}{2}+\e}\big) $$
with the norm
\begin{align*}
\|u\|_{Y(t)}:= &\sup_{0\le \tau \le t} \Big( \ell(\tau)\|u(\tau,\cdot)\|_{L^2} + (1+\tau)^{\frac{n-1+\e}{4}}\big\|\nabla^{\frac{n}{2}+\e} u(\tau,\cdot)\big\|_{L^2} \\ 
&\hspace{2cm}+ (1+\tau)^{\frac{n}{8}}\|u_t(\tau,\cdot)\|_{L^2}+ (1+\tau)^{\frac{n+\e}{4}}\big\|\nabla^{\frac{n}{2}+\e} u_t(\tau,\cdot)\big\|_{L^2}\Big),
\end{align*}
where $\ell(\tau)$ is defined in the previous section. In the sequel, we follow the strategy in the previous section. Then, we have the integral equation corresponding to \eqref{equation4.1} in the following form:
$$ u(t,x)= u^{\text{ln}}(t,x) + \int_0^t \mathcal{K}_1(t-\tau,x) \ast_x \big(a \circ \nabla |u_t(\tau,x)|^p\big) d\tau=: u^{\text{ln}}(t,x)+ \tilde{u}^{\text{nl}}(t,x). $$
We define a mapping $\Phi: \,\, Y(t) \longrightarrow Y(t)$ in the following way:
$$ \Phi[u](t,x)= u^{\text{ln}}(t,x)+ \tilde{u}^{\text{nl}}(t,x). $$
As we see in the proof of Theorem \ref{theorem3.1}, the proof of Theorem \ref{theorem4.1} is reduced to prove the following estimates:
\begin{align}
\|\Phi[u]\|_{Y(t)}& \lesssim \|(u_0,u_1)\|_{\mathcal{B}}+ \|u\|^p_{Y(t)}, \label{inequality4.1} \\
\|\Phi[u]- \Phi[v]\|_{Y(t)}& \lesssim \|u-v\|_{Y(t)} \big(\|u\|^{p-1}_{Y(t)}+ \|v\|^{p-1}_{Y(t)}\big). \label{inequality4.2}
\end{align}
Before indicating the both above inequalities, the application of Proposition \ref{SobolevEmbedding} gives
\begin{align*}
\|u_t(\tau,\cdot) \|_{L^\ity} &\lesssim \|u_t(\tau,\cdot) \|^{1-\theta}_{L^2}\,\|u_t(\tau,\cdot) \|^{\theta}_{\dot{H}^{\frac{n}{2}+\varepsilon}}\qquad \text{with }\theta= \frac{\frac{n}{2}}{\frac{n}{2}+\e} \\
&\lesssim (1+\tau)^{-\frac{n}{4}} \|u\|_{Y(t)}.
\end{align*}
By the same fashion as in the proof of Theorem \ref{theorem3.1}, we obtain the following auxiliary estimates for any $p\ge 2$:
\begin{align*}
\big\||u_t(\tau,\cdot)|^p\big\|_{L^1}= \|u_t(\tau,\cdot)\|^p_{L^p} &\lesssim (1+\tau)^{-\frac{n}{4}(p-1)}\|u\|^p_{Y(\tau)}, \\ 
\big\||u_t(\tau,\cdot)|^p\big\|_{L^2}= \|u_t(\tau,\cdot)\|^p_{L^{2p}} &\lesssim (1+\tau)^{-\frac{n}{4}(p-\frac{1}{2})}\|u\|^p_{Y(\tau)}, \\
\big\|\nabla |u_t(\tau,\cdot)|^p\big\|_{L^2} &\lesssim (1+\tau)^{-\frac{n}{4}(p-\frac{1}{2})-\frac{1}{4}}\|u\|^p_{Y(\tau)}.
\end{align*}
\par \textit{First let us prove the inequality (\ref{inequality4.1}).} As in the proof of the estimate \eqref{inequality3.1}, 
we only show the estimate
\begin{equation}
\|\tilde{u}^{\text{nl}}\|_{Y(t)} \lesssim \|u\|^p_{Y(t)}. \label{inequality4.3}
\end{equation}
We will follow by same method as in the previous section. Our proof is divided into two steps. \medskip

Step 1: $\quad$ We may estimate the norm $\|\tilde{u}^{\text{nl}}(t,\cdot)\|_{L^2}$ as follows:
\begin{align*}
\|\tilde{u}^{\text{nl}}(t,\cdot)\|_{L^2} &\lesssim \int_0^t \big\|\nabla \mathcal{K}_{1L}(t-\tau,x) \ast_x |u_t(\tau,x)|^p\big\|_{L^2} d\tau \\
&\qquad+ \int_0^t \big\|\nabla \big( \mathcal{K}_{1M}(t-\tau,x) + \mathcal{K}_{1H}(t-\tau,x) \big) \ast_x |u_t(\tau,x)|^p\big\|_{L^2} d\tau
\end{align*}
\begin{align*}
&\lesssim \int_0^{t/2} (1+t-\tau)^{-\frac{n}{8}}\big\||u_t(\tau,\cdot)|^p\big\|_{L^1} d\tau+ \int_{t/2}^t \big\||u_t(\tau,\cdot)|^p\big\|_{L^2} d\tau \\
&\qquad+ \int_0^t (1+t-\tau)^{-\frac{3}{2}}\big\||u_t(\tau,\cdot)|^p\big\|_{L^2} d\tau \\
&\lesssim \Big(\int_0^{t/2} (1+t-\tau)^{-\frac{n}{8}}(1+\tau)^{-\frac{n}{4}(p-1)} d\tau+ \int_{t/2}^t (1+\tau)^{-\frac{n}{4}(p-\frac{1}{2})} d\tau \\
&\qquad +\int_0^t (1+t-\tau)^{-\frac{3}{2}}(1+\tau)^{-\frac{n}{4}(p-\frac{1}{2})} d\tau \Big)\|u\|^p_{Y(t)}.
\end{align*}
When $p \ge 1+\frac{4}{n}$, it follows immediately $-\frac{n}{4}(p-1) \le -1$. 
On the other hand, if $1+\frac{3}{n} \le p <1+\frac{4}{n}$, we see $-1< -\frac{n}{4}(p-1) \le -\frac{3}{4}$. 
Hence, using the relations
$$ \begin{cases}
1+t- \tau \approx 1+t &\text{ if }\,\tau \in [0,t/2] \\
1+\tau \approx 1+t &\text{ if }\,\tau \in [t/2,t]
\end{cases} $$
to control the first two integrals we derive
\begin{align*}
\int_0^{t/2} (1+t-\tau)^{-\frac{n}{8}}(1+\tau)^{-\frac{n}{4}(p-1)} d\tau &\lesssim 
\begin{cases}
(1+t)^{-\frac{n}{8}} \log(t+e) &\text{ if } p \ge 1+\frac{4}{n} \\ 
(1+t)^{-\frac{n}{8}+\frac{1}{4}} &\text{ if } 1+\frac{3}{n} \le p <1+\frac{4}{n}
\end{cases} \\
&\lesssim (1+t)^{-\frac{n}{8}+\frac{1}{4}}
\end{align*}
for any $p\ge 1+\frac{3}{n}$, and
$$ \int_{t/2}^t (1+\tau)^{-\frac{n}{4}(p-\frac{1}{2})} d\tau \lesssim (1+t)^{-\frac{n}{8}+\frac{1}{4}}, $$
where we used the fact that $-\frac{n}{4}(p-\frac{1}{2})+1 \le -\frac{n}{8}+\frac{1}{4}$ since the condition $p\ge 1+\frac{3}{n}$ holds from \eqref{exponent4.1}. After applying Lemma \ref{LemmaIntegral.1}, we arrive at the following estimate for the third integral:
$$ \int_0^t (1+t-\tau)^{-\frac{3}{2}}(1+\tau)^{-\frac{n}{4}(p-\frac{1}{2})} d\tau \lesssim (1+t)^{-\min\big\{\frac{3}{2},\,\frac{n}{4}(p-\frac{1}{2})\big\}} \lesssim (1+t)^{-\frac{n}{8}+\frac{1}{4}}. $$
Therefore, we have proved that
$$ \|\tilde{u}^{\text{nl}}(t,\cdot)\|_{L^2} \lesssim (1+t)^{-\frac{n}{8}+\frac{1}{4}} \|u\|^p_{Y(t)} \lesssim
\begin{cases}
\sqrt{\log(t+e)} \|u\|^p_{Y(t)} &\text{ if }\,n=2, \\
(1+t)^{-\frac{n}{8}+ \frac{1}{4}} \|u\|^p_{Y(t)} &\text{ if }\,n=3,4.
\end{cases} $$
Step 2: $\quad$ By using the same ideas, we may control the remaining norms
$$ \big\|\nabla^{\frac{n}{2}+\e} \tilde{u}^{\text{nl}}(t,\cdot)\big\|_{L^2}, \quad \|\tilde{u}_t^{\text{nl}}(t,\cdot)\|_{L^2} \quad \text{ and }\quad \big\|\nabla^{\frac{n}{2}+\e} \tilde{u}^{\text{nl}}_t(t,\cdot)\big\|_{L^2} $$
as follows:
\begin{align*}
\big\|\nabla^{\frac{n}{2}+\e} \tilde{u}^{\text{nl}}(t,\cdot)\big\|_{L^2} &\lesssim \int_0^t \big\|\nabla^{\frac{n}{2}+1+\e} \mathcal{K}_{1L}(t-\tau,x) \ast_x |u_t(\tau,x)|^p\big\|_{L^2} d\tau \\
&\qquad+ \int_0^t \big\|\nabla^{\frac{n}{2}+\e} \big( \mathcal{K}_{1M}(t-\tau,x) + \mathcal{K}_{1H}(t-\tau,x) \big) \ast_x \nabla |u_t(\tau,x)|^p\big\|_{L^2} d\tau
\end{align*}
\begin{align*}
&\lesssim \int_0^{t/2} (1+t-\tau)^{-\frac{n+\e}{4}}\big\||u_t(\tau,\cdot)|^p\big\|_{L^1} d\tau \\
&\qquad+ \int_{t/2}^t (1+t-\tau)^{-\frac{1}{4}(\frac{n}{2}+\e)}\big\||u_t(\tau,\cdot)|^p\big\|_{L^2} d\tau \\
&\qquad+ \int_0^t (1+t-\tau)^{-2+\frac{n}{4}+\frac{\e}{2}}\big\|\nabla |u_t(\tau,\cdot)|^p\big\|_{L^2} d\tau \\
&\lesssim \Big(\int_0^{t/2} (1+t-\tau)^{-\frac{n+\e}{4}}(1+\tau)^{-\frac{n}{4}(p-1)} d\tau \\
&\qquad+ \int_{t/2}^t (1+t-\tau)^{-\frac{1}{4}(\frac{n}{2}+\e)}(1+\tau)^{-\frac{n}{4}(p-\frac{1}{2})} d\tau \\
&\qquad+ \int_0^t (1+t-\tau)^{-2+\frac{n}{4}+\frac{\e}{2}}(1+\tau)^{-\frac{n}{4}(p-\frac{1}{2})-\frac{1}{4}} d\tau \Big)\|u\|^p_{Y(t)}.
\end{align*}
By the similar way to Step 1, we gain
$$ \big\|\nabla^{\frac{n}{2}+\e} \tilde{u}^{\text{nl}}(t,\cdot)\big\|_{L^2} \lesssim (1+t)^{-\frac{n-1+\e}{4}} \|u\|^p_{Y(t)}. $$
Furthermore, one gets
\begin{align*}
\|\tilde{u}_t^{\text{nl}}(t,\cdot)\|_{L^2} &\lesssim \int_0^t \big\|\partial_t\nabla \mathcal{K}_{1L}(t-\tau,x) \ast_x |u_t(\tau,x)|^p\big\|_{L^2} d\tau \\
&\qquad+ \int_0^t \big\|\partial_t\nabla \big( \mathcal{K}_{1M}(t-\tau,x) + \mathcal{K}_{1H}(t-\tau,x) \big) \ast_x |u_t(\tau,x)|^p\big\|_{L^2} d\tau \\
&\lesssim \int_0^{t/2} (1+t-\tau)^{-\frac{n}{8}-\frac{1}{4}}\big\||u_t(\tau,\cdot)|^p\big\|_{L^1} d\tau+ \int_{t/2}^t (1+t-\tau)^{-\frac{1}{4}}\big\||u_t(\tau,\cdot)|^p\big\|_{L^2} d\tau \\
&\qquad+ \int_0^t e^{-c(t-\tau)} \big\|\nabla |u_t(\tau,\cdot)|^p\big\|_{L^2} d\tau+ \int_0^t (1+t-\tau)^{-\frac{5}{2}} \big\||u_t(\tau,\cdot)|^p\big\|_{L^2} d\tau \\
&\lesssim \Big(\int_0^{t/2} (1+t-\tau)^{-\frac{n}{8}-\frac{1}{4}}(1+\tau)^{-\frac{n}{4}(p-1)} d\tau \\
&\qquad+ \int_{t/2}^t (1+t-\tau)^{-\frac{1}{4}}(1+\tau)^{-\frac{n}{4}(p-\frac{1}{2})} d\tau \\
&\qquad+ \int_0^t e^{-c(t-\tau)}(1+\tau)^{-\frac{n}{4}(p-\frac{1}{2})-\frac{1}{4}} d\tau \\
&\qquad+ \int_0^t (1+t-\tau)^{-\frac{5}{2}} (1+\tau)^{-\frac{n}{4}(p-\frac{1}{2})} d\tau \Big)\|u\|^p_{Y(t)}
\end{align*}
Then, an analogous treatment as we estimated in Step 1 leads to
\begin{align*}
\int_0^{t/2} (1+t-\tau)^{-\frac{n}{8}-\frac{1}{4}}(1+\tau)^{-\frac{n}{4}(p-1)} d\tau &\lesssim (1+t)^{-\frac{n}{8}} \\
\int_{t/2}^t (1+t-\tau)^{-\frac{1}{4}}(1+\tau)^{-\frac{n}{4}(p-\frac{1}{2})} d\tau &\lesssim (1+t)^{-\frac{n}{8}} \\
\int_0^t (1+t-\tau)^{-\frac{5}{2}} (1+\tau)^{-\frac{n}{4}(p-\frac{1}{2})} d\tau &\lesssim (1+t)^{-\frac{n}{8}}.
\end{align*}
After employing Lemma \ref{LemmaIntegral.2}, one has
\begin{align*}
\int_0^t e^{-c(t-\tau)}(1+\tau)^{-\frac{n}{4}(p-\frac{1}{2})-\frac{1}{4}} d\tau \lesssim (1+t)^{-\frac{n}{4}(p-\frac{1}{2})-\frac{1}{4}} &\lesssim (1+t)^{-\frac{n}{8}-1} \\
&\lesssim (1+t)^{-\frac{n}{8}}
\end{align*}
because of the hypothesis (\ref{exponent4.1}). All the above estimates follow that
$$ \|\tilde{u}_t^{\text{nl}}(t,\cdot)\|_{L^2} \lesssim (1+t)^{-\frac{n}{8}} \|u\|^p_{Y(t)}. $$
Now let us control the norm $\big\|\nabla^{\frac{n}{2}+\e} \tilde{u}_t^{\text{nl}}(t,\cdot)\big\|_{L^2}$ in the following way:
\begin{align*}
\big\|\nabla^{\frac{n}{2}+\e} \tilde{u}_t^{\text{nl}}(t,\cdot)\big\|_{L^2} &\lesssim \int_0^t \big\|\partial_t\nabla^{\frac{n}{2}+1+\e} \mathcal{K}_{1L}(t-\tau,x) \ast_x |u_t(\tau,x)|^p\big\|_{L^2} d\tau \\
&\qquad+ \int_0^t \big\|\partial_t \nabla^{\frac{n}{2}+\e} \big( \mathcal{K}_{1M}(t-\tau,x) + \mathcal{K}_{1H}(t-\tau,x) \big) \ast_x \nabla|u_t(\tau,x)|^p\big\|_{L^2} d\tau \\
&\lesssim \int_0^{t/2} (1+t-\tau)^{-\frac{n+1+\e}{4}}\big\||u_t(\tau,\cdot)|^p\big\|_{L^1} d\tau \\
&\qquad+ \int_{t/2}^t (1+t-\tau)^{-\frac{1}{4}(\frac{n}{2}+1+\e)}\big\||u_t(\tau,\cdot)|^p\big\|_{L^2} d\tau \\
&\qquad+ \int_0^t e^{-c(t-\tau)} (t-\tau)^{-\frac{1}{4}(\frac{n}{2}+1+\varepsilon)} \big\||u_t(\tau,\cdot)|^p\big\|_{L^2} d\tau \\
&\qquad+ \int_0^t (1+t-\tau)^{-3+\frac{n}{4}+\frac{\e}{2}} \big\|\nabla |u_t(\tau,\cdot)|^p\big\|_{L^2} d\tau \\
&\lesssim \Big(\int_0^{t/2} (1+t-\tau)^{-\frac{n+1+\e}{4}}(1+\tau)^{-\frac{n}{4}(p-1)} d\tau \\
&\qquad+ \int_{t/2}^t (1+t-\tau)^{-\frac{1}{4}(\frac{n}{2}+1+\e)}(1+\tau)^{-\frac{n}{4}(p-\frac{1}{2})} d\tau \\
&\qquad+ \int_0^t  e^{-c(t-\tau)} (t-\tau)^{-\frac{1}{4}(\frac{n}{2}+1+\varepsilon)} (1+\tau)^{-\frac{n}{4}(p-\frac{1}{2})} d\tau \\
&\qquad+ \int_0^t (1+t-\tau)^{-3+\frac{n}{4}+\frac{\e}{2}} (1+\tau)^{-\frac{n}{4}(p-\frac{1}{2})-\frac{1}{4}} d\tau \Big)\|u\|^p_{Y(t)},
\end{align*}
where we used the estimate \eqref{lemma2.5.4} from Lemma \ref{lemma2.2}. Analogously to the estimation for $\|\tilde{u}_t^{\text{nl}}(t,\cdot)\|_{L^2}$, we may derive
$$ \big\|\nabla^{\frac{n}{2}+\e} \tilde{u}_t^{\text{nl}}(t,\cdot)\big\|_{L^2} \lesssim (1+t)^{-\frac{n+\e}{4}} \|u\|^p_{Y(t)}. $$
Therefore, from the definition of the norm in $Y(t)$ we obtain immediately the inequality (\ref{inequality4.3}).
\par \textit{Next let us prove the inequality \eqref{inequality4.2}.} We shall follow the strategy used in the proof of the inequality (\ref{inequality4.3}). The new difficulty is to require the estimates for the term
$$|u_t(\tau,\cdot)|^p- |v_t(\tau,\cdot)|^p $$
in $L^1$, $L^2$ and $\dot{H}^{1}$. Then, repeating an analogous treatment as in the proof of the inequality (\ref{inequality4.3}) we may conlcude the inequality (\ref{inequality4.2}). Indeed, by using H\"{o}lder's inequality we get
\begin{align*}
\big\||u_t(\tau,\cdot)|^p-|v_t(\tau,\cdot)|^p\big\|_{L^1} &\lesssim \|u_t(\tau,\cdot)- v_t(\tau,\cdot)\|_{L^{p}} \big(\|u_t(\tau,\cdot)\|^{p-1}_{L^{p}}+\|v_t(\tau,\cdot)\|^{p-1}_{L^{p}}\big) \\ 
\big\||u_t(\tau,\cdot)|^p- |v_t(\tau,\cdot)|^p\big\|_{L^2} &\lesssim \|u_t(\tau,\cdot)- v_t(\tau,\cdot)\|_{L^{2p}} \big(\|u_t(\tau,\cdot)\|^{p-1}_{L^{2p}}+\|v_t(\tau,\cdot)\|^{p-1}_{L^{2p}}\big).
\end{align*}
Analogously to the proof of (\ref{inequality4.3}), applying Proposition \ref{fractionalGagliardoNirenberg} to the norms
$$ \|u_t(\tau,\cdot)- v_t(\tau,\cdot)\|_{L^\eta}, \quad \|u_t(\tau,\cdot)\|_{L^\eta}, \quad \|v_t(\tau,\cdot)\|_{L^\eta} $$
with $\eta=p$ and $\eta=2p$ we may arrive at the following estimates:
\begin{align*}
\big\||u_t(\tau,\cdot)|^p-|v_t(\tau,\cdot)|^p\big\|_{L^1} &\lesssim (1+\tau)^{-\frac{n}{4}(p-1)} \|u-v\|_{Y(t)}\big(\|u\|^{p-1}_{Y(t)}+\|v\|^{p-1}_{Y(t)}\big),\\
\big\||u_t(\tau,\cdot)|^p- |v_t(\tau,\cdot)|^p\big\|_{L^2} &\lesssim (1+\tau)^{-\frac{n}{4}(p-\frac{1}{2})} \|u-v\|_{Y(t)}\big(\|u\|^{p-1}_{Y(t)}+\|v\|^{p-1}_{Y(t)}\big).
\end{align*}
We also obtain the following estimates as \eqref{ineq:3.4}:
\begin{equation*}
\begin{split}
&\big\| \nabla \big(|u_{t}(\tau,\cdot)|^p-|v_{t}(\tau,\cdot)|^p\big) \big\|_{L^2} \\
&\qquad \le C \|u_{t}(\tau, \cdot) \|_{L^\infty}^{p-1}  \|\nabla u_{t}(\tau, \cdot)  - \nabla v_{t}(\tau, \cdot) \|_{L^2} \\
&\qquad\quad + C  \|\nabla v_{t}(\tau, \cdot) \|_{L^2} \big(\|u_{t}(\tau, \cdot) \|_{L^\infty}^{p-2} 
+\|v_{t}(\tau, \cdot) \|_{L^\infty}^{p-2}\big) \|u_{t}(\tau, \cdot) -v_{t}(\tau, \cdot) \|_{L^\infty},
\end{split}
\end{equation*}
which leads the estimate 
\begin{equation} 
\begin{split}
\big\|\nabla \big(|u_{t}(\tau,\cdot)|^p-|v_{t}(\tau,\cdot)|^p\big) \big\|_{L^2} &\lesssim
(1+\tau)^{-\frac{n}{4}(p-\frac{1}{2})-\frac{1}{4}} \|u-v\|_{Y(\tau)}\big(\|u\|^{p-1}_{Y(\tau)}+\|v\|^{p-1}_{Y(\tau)}\big). 
\end{split}
\end{equation}
This completes the proof of inequality (\ref{inequality4.2}).
\end{proof}

\section{Further discussions} \label{Further.Sec}
\subsection{Large time behavior of global solutions}
This subsection is to discuss the large time behavior of the derived global solutions to (\ref{equation3.1}) and (\ref{equation4.1}) in Theorems \ref{theorem3.1} and \ref{theorem4.1}, respectively. Throughout this subsection, we denote some quantities and some kernels as follows:
\begin{align*}
P_j &:= \int_{\R^n} u_j(y)dy \quad \text{ with }j=0,1,\\ 
\mathcal{H}_0(t,x) &:= \mathcal{F}^{-1}_{\xi \to x}\Big(e^{-\frac{t|\xi|^4}{2}}\cos(t|\xi|)\Big)(t,x), \\
\mathcal{H}_1(t,x) &:= \mathcal{F}^{-1}_{\xi \to x}\Big(e^{-\frac{t|\xi|^4}{2}}\frac{\sin(t|\xi|)}{|\xi|}\Big)(t,x).
\end{align*}
Let us return to the Cauchy problems (\ref{equation3.1}) and (\ref{equation4.1}) in the following common form:
\begin{equation}
\begin{cases}
u_{tt}- \Delta u+ \nu (-\Delta)^2 u_t= a \circ \nabla \big|\partial_t^j u\big|^p, &\quad x\in \R^n,\, t > 0, \\
u(0,x)= u_0(x),\quad u_t(0,x)=u_1(x), &\quad x\in \R^n,
\end{cases}
\label{equation3-4.1}
\end{equation}
where $j=0,1$. We intend to prove the large time behavior of global solutions to (\ref{equation3-4.1}) in the following result.
\begin{theorem} \label{theorem6.1}
Let $\e$ is a sufficiently small positive constant. We assume that the exponent $p$ and the space dimension $n$ satisfy the following conditions:
\begin{equation} \label{exponent6.1}
\begin{cases}
p> 1+\dfrac{4}{n-1} \quad \text{ and }\quad n=2,3,4,5 &\text{ if }\, j=0, \\
p> 1+\dfrac{3}{n} \quad \text{ and }\quad n=2,3 &\text{ if }\, j=1, \\
p\ge 2 \quad \text{ and }\quad n=4 &\text{ if }\, j=1.
\end{cases}
\end{equation}
Moreover, we choose the initial data
$$ \begin{cases}
(u_0,u_1) \in \big(L^{1,1} \cap H^{\frac{n}{2}+\e}\big) \times \big(L^{1,1}\cap H^{\frac{n}{2}-1+\e}\big) &\text{ if }\, j=0, \\
(u_0,u_1) \in \big(L^{1,1} \cap H^{\frac{n}{2}+1+\e}\big) \times \big(L^{1,1}\cap H^{\frac{n}{2}+\e}\big) &\text{ if }\, j=1.
\end{cases} $$
Then, the global (in time) small data solutions to (\ref{equation3-4.1}) enjoy the following estimate for $t\gg 1$:
\begin{equation}
\big\|u(t,\cdot)- P_0\,\mathcal{H}_0(t,\cdot)- P_1\,\mathcal{H}_1(t,\cdot) \big\|_{L^2}= o\big(t^{-\frac{n}{8}+ \frac{1}{4}}\big). \label{LargetimeEstimate6.1}
\end{equation}
\end{theorem}
In order to prove our main result in this section, we need the following auxiliary estimates.
\begin{proposition}[Theorem 1.3 in \cite{IkehataIyota}] \label{Useful.Prop}
Let $1\le n\le 5$ and $\ell\ge 1$. Let us assume $(u_0,u_1) \in \big(L^{1,1} \cap H^\ell\big) \times \big(L^{1,1}\cap H^{\ell-1}\big)$ in (\ref{equation1.2}). Then, the solutions to (\ref{equation1.2}) satisfy the following estimate for $t\gg 1$:
\begin{align}
&\big\|u(t,\cdot)- P_0\,\mathcal{H}_0(t,\cdot)- P_1\,\mathcal{H}_1(t,\cdot)\big\|_{L^2} \nonumber \\ 
&\qquad \lesssim t^{-\frac{n}{8}-\frac{1}{4}}\|u_0\|_{L^{1,1}}+ t^{-\frac{n}{8}}\|u_1\|_{L^{1,1}}+ e^{-ct}\|(u_0,u_1)\|_{(L^1\cap L^2)\times (L^1\cap L^2)} \nonumber \\
&\hspace{5.7cm}+ t^{-\frac{\ell}{2}}\big(\|u_0\|_{H^\ell}+\|u_1\|_{H^{\ell-1}}\big), \label{prop.6.1}
\end{align}
where $c$ is a suitable positive constant.
\end{proposition}

\begin{proof} In order to show the proof of Theorem \ref{theorem6.1}, let us consider two cases including $j=0$ and $j=1$ individually.
\begin{itemize}[leftmargin=*]
\item \textit{Case 1}: If $j=0$, then we take $\ell= \frac{n}{2}+\e$ in Proposition \ref{Useful.Prop}. By virtue of the statement (\ref{prop.6.1}), to indicate the desired estimate (\ref{LargetimeEstimate6.1}), we need to show the following estimate instead:
\begin{equation}
\Big\|\int_0^t \mathcal{K}_1(t-\tau,x) \ast_x \nabla |u(\tau,x)|^p d\tau \Big\|_{L^2}= o\big(t^{-\frac{n}{8}+ \frac{1}{4}}\big) \label{the.6.1.1}
\end{equation}
by using the representation of solutions $u(t,x)= u^{\text{ln}}(t,x)+u^{\text{nl}}(t,x)$ to (\ref{equation3.1}) as in the proof of Theorem \ref{theorem3.1}. 
Now we assume $n=2,3,4$. 
First of all, recalling the proof of Theorem \ref{theorem3.1} we have achieved the following estimates:
\begin{align}
\big\||u(\tau,\cdot)|^p\big\|_{L^1}&\lesssim (1+\tau)^{-\frac{n}{4}(p-1)+\frac{p}{4}+\epsilon_{0}}, \label{the.6.1.1*} \\ 
\big\||u(\tau,\cdot)|^p\big\|_{L^2}&\lesssim (1+\tau)^{-\frac{n}{4}(p-\frac{1}{2})+\frac{p}{4}+\epsilon_{0}},\label{the.6.1.2*}
\end{align}
where $\epsilon_0$ is given by \eqref{eq:3.3}.
Similarly to the strategy which we have used in the proof of Theorem \ref{theorem3.1}, we separate the left-hand side term of (\ref{the.6.1.1}) into several parts as follows:
\begin{align*}
&\Big\|\int_0^t \mathcal{K}_1(t-\tau,x) \ast_x \nabla |u(\tau,x)|^p d\tau \Big\|_{L^2} \\
&\qquad \lesssim \int_0^t \big\|\nabla \mathcal{K}_{1L}(t-\tau,x) \ast_x |u(\tau,x)|^p\big\|_{L^2} d\tau \\
&\qquad\quad+ \int_0^t \big\|\nabla \big( \mathcal{K}_{1M}(t-\tau,x) + \mathcal{K}_{1H}(t-\tau,x) \big)
\ast_x |u(\tau,x)|^p\big\|_{L^2} d\tau \\
&\qquad \lesssim \int_0^{t/2} (1+t-\tau)^{-\frac{n}{8}}\big\||u(\tau,\cdot)|^p\big\|_{L^1} d\tau+ \int_{t/2}^t \big\||u(\tau,\cdot)|^p\big\|_{L^2} d\tau \\
&\qquad\quad+ \int_0^t (1+t-\tau)^{-\frac{3}{2}}\big\||u(\tau,\cdot)|^p\big\|_{L^2} d\tau \\
&\qquad \lesssim \int_0^{t/2} (1+t-\tau)^{-\frac{n}{8}}(1+\tau)^{-\frac{n}{4}(p-1)+\frac{p}{4}+\epsilon_{0}} d\tau+ \int_{t/2}^t (1+\tau)^{-\frac{n}{4}(p-\frac{1}{2})+\frac{p}{4}+\epsilon_{0}} d\tau \\
&\qquad\quad+\int_0^t (1+t-\tau)^{-\frac{3}{2}}(1+\tau)^{-\frac{n}{4}(p-\frac{1}{2})+\frac{p}{4}+\epsilon_{0}} d\tau.
\end{align*}
In addition, the condition (\ref{exponent6.1}) follows immediately $-\frac{n}{4}(p-1)+\frac{p}{4}< -\frac{3}{4}$. Thus, it implies
$$ -\frac{n}{4}(p-1)+\frac{p}{4}+\epsilon_{0}< -\frac{3}{4} \quad \text{ and }\quad -\frac{n}{4}\Big(p-\frac{1}{2}\Big)+ \frac{p}{4}+\epsilon_{0}< -\frac{n}{8}-\frac{3}{4},  $$
which lead to
\begin{align*}
&\int_0^{t/2} (1+t-\tau)^{-\frac{n}{8}}(1+\tau)^{-\frac{n}{4}(p-1)+\frac{p}{4}+\epsilon_{0}} d\tau \\
&\qquad \lesssim (1+t)^{-\frac{n}{8}}\big(1+ \log(e+t)+ (1+t)^{-\frac{n}{4}(p-1)+\frac{p}{4}+\epsilon_{0}+1}\big) \lesssim (1+t)^{-\frac{n}{8}+\frac{1}{4}-\epsilon}
\end{align*}
and
$$\int_{t/2}^t (1+\tau)^{-\frac{n}{4}(p-\frac{1}{2})+\frac{p}{4}+\epsilon_{0}} d\tau \lesssim (1+t)^{-\frac{n}{4}(p-\frac{1}{2})+\frac{p}{4}+1+\epsilon_{0}} \lesssim (1+t)^{-\frac{n}{8}+\frac{1}{4}-\epsilon} $$
for some sufficiently small constant $\epsilon>0$. Moreover, as indicated in the proof of Theorem \ref{theorem3.1}, one derives
$$ \int_0^t (1+t-\tau)^{-\frac{3}{2}}(1+\tau)^{-\frac{n}{4}(p-\frac{1}{2})+\frac{p}{4}+\epsilon_{0}} d\tau \lesssim  (1+t)^{-\frac{n}{8}-\frac{3}{4}}. $$
Summing up all the above estimates gives the estimate (\ref{the.6.1.1}) what we wanted to prove.
Next we show the estimate \eqref{the.6.1.1} for $n=5$.
Applying the estimates \eqref{eq:3.10} and \eqref{eq:3.11}, instead of \eqref{the.6.1.1*} and \eqref{the.6.1.2*},
we have 
\begin{align*}
&\Big\|\int_0^t \mathcal{K}_1(t-\tau,x) \ast_x \nabla |u(\tau,x)|^p d\tau \Big\|_{L^2} \\
&\qquad\lesssim \int_0^t \big\|\nabla \mathcal{K}_{1L}(t-\tau,x) \ast_x |u(\tau,x)|^p\big\|_{L^2} d\tau \\
&\qquad\quad+ \int_0^t \big\|\nabla \big( \mathcal{K}_{1M}(t-\tau,x) + \mathcal{K}_{1H}(t-\tau,x) \big)
\ast_x |u(\tau,x)|^p\big\|_{L^2} d\tau \\
&\qquad\lesssim \int_0^{t} (1+t-\tau)^{-\frac{5}{8}}\big\||u(\tau,\cdot)|^p\big\|_{L^1} d\tau+ \int_0^t (1+t-\tau)^{-\frac{3}{2}}\big\||u(\tau,\cdot)|^p\big\|_{L^2} d\tau \\
&\qquad\lesssim \int_0^{t} (1+t-\tau)^{-\frac{5}{8}}(1+\tau)^{-\frac{3}{4}(p-1)+\e_{2} (p-2)} d\tau \\
&\qquad\quad +\int_0^t (1+t-\tau)^{-\frac{1}{2}}(1+\tau)^{-\frac{3}{4}(p-\frac{1}{2})+\e_{2} (p-1)} d\tau.
\end{align*}
Therefore, by using Lemma \ref{LemmaIntegral.1} it holds
\begin{equation} \label{eq:5.8}
\begin{split}
& \int_0^{t} (1+t-\tau)^{-\frac{5}{8}}(1+\tau)^{-\frac{3}{4}(p-1)+\e_{2} (p-2)} d\tau \\
&\qquad \lesssim 
\begin{cases}
(1+t)^{-\frac{5}{8}} &\text{if}\ \frac{3}{4}(p-1)-\e_{2} (p-2)>1, \\
(1+t)^{-\frac{5}{8}} \log(e+t) &\text{if}\ \frac{3}{4}(p-1)-\e_{2} (p-2)=1, \\  
(1+t)^{-\frac{3}{4}(p-1)+\e_{2} (p-2)+\frac{3}{8}} &\text{if}\ \frac{3}{4}(p-1)-\e_{2} (p-2)<1
\end{cases} \\
&\qquad =o(t^{-\frac{3}{8}})
\end{split}
\end{equation}
as $t \to \infty$, because of the smallness of $\varepsilon>0$ and the fact that
$$ -\frac{3}{4}(p-1)+\e_{2}(p-2) +\frac{3}{8}< -\frac{3}{8}$$
under the assumption \eqref{exponent6.1}.
On the other hand, thanks to Lemma \ref{LemmaIntegral.1}, we also have 
\begin{align}
\int_0^t (1+t-\tau)^{-\frac{3}{2}}(1+\tau)^{-\frac{3}{4}(p-\frac{1}{2})+\e_{2} (p-1)} d\tau &\lesssim (1+t)^{-\min\big\{\frac{3}{2},\frac{3}{4}(p-\frac{1}{2})-\e_{2} (p-1)\big\}} \nonumber \\
&= o(t^{-\frac{3}{8}}) \label{eq:5.9}
\end{align}
as $t \to \infty$, where we have used the smallness of $\varepsilon>0$ and the assumption \eqref{exponent6.1}, again.
Summing up \eqref{eq:5.8} and \eqref{eq:5.9} we obtain the estimate \eqref{the.6.1.1} for $n=5$.\\
\item \textit{Case 2}: If $j=1$, then we take $\ell= \frac{n}{2}+1+\e$ in Proposition \ref{Useful.Prop}. Following the proof of Case 1, we may conclude the proof of Case 2 by the aid of the auxiliary estimates as follows:
\begin{align*}
\big\||u_t(\tau,\cdot)|^p\big\|_{L^1}&\lesssim (1+\tau)^{-\frac{n}{4}(p-1)}, \\
\big\||u_t(\tau,\cdot)|^p\big\|_{L^2}&\lesssim (1+\tau)^{-\frac{n}{4}(p-\frac{1}{2})},
\end{align*}
which we have obtained from the proof of Theorem \ref{theorem4.1}.
\end{itemize}
Conclusion, our proof is completed.
\end{proof}

\begin{remark}
\fontshape{n}
\selectfont
We want to point out that Theorem \ref{theorem6.1} is concerned with the large time behavior of global derived solutions to (\ref{equation3.1}) and (\ref{equation4.1}) in the supercritical cases only, i.e. $p> 1+\dfrac{4}{n-1}$ for all space dimensions $n=2,3,4,5$ and $p> 1+\dfrac{3}{n}$ for all space dimensions $n=2,3$, respectively. It remains an open problem to explore such result in the critical cases, i.e. $p= 1+\dfrac{4}{n-1}$ for $n=3,4,5$ to (\ref{equation3.1}) and $p= 1+\dfrac{3}{n}$ for $n=2,3$ to (\ref{equation4.1}). The fact is that the main difficulty lies on dealing with the integrals
$$ \int_0^s (1+\rho)^{-\frac{n}{4}(p-1)+\frac{p}{4}} d\rho \quad \text{ and }\quad \int_0^s (1+\rho)^{-\frac{n}{4}(p-1)} d\rho \quad \text{ for } s\gg 1, $$
which are not infinitesimal quantities of $s^{\frac{1}{4}}$ when $p= 1+\dfrac{4}{n-1}$ and $p= 1+\dfrac{3}{n}$, respectively. 
\end{remark}

\subsection{Mixed nonlinearities}
In this subsection, relying on the proof of Theorems \ref{theorem3.1} and \ref{theorem4.1} one may catch the global (in time) existence of small data solutions and their decay properties to (\ref{equation1.1}) with the nonlinear function $f(u,u_t)= \big|\partial_t^j u\big|^p+ a \circ \nabla \big|\partial_t^j u\big|^q$ with $j=0,1$, i.e. the following semi-linear equations with mixing two different kinds of nonlinearities:
\begin{equation}
\begin{cases}
u_{tt}- \Delta u+ \nu (-\Delta)^2 u_t= \big|\partial_t^j u\big|^p+ a \circ \nabla \big|\partial_t^j u\big|^q, &\quad x\in \R^n,\, t > 0, \\
u(0,x)= u_0(x),\quad u_t(0,x)=u_1(x), &\quad x\in \R^n,
\end{cases}
\label{equation3.3}
\end{equation}
where $p,q>1$. We obtain the following results.
\begin{theorem}[Equation \eqref{equation3.3} with $j=0$] \label{theorem3.31}
Let $n=2,3,4,5$ and $\e$ is a sufficiently small positive constant. Assume that the following conditions for $p$ and $q$ hold:
\begin{equation} \label{exponent3.31}
\begin{split}
&p> 6, \hspace{2.2cm} q> 5 \hspace{2.2cm} \text{if}\ n=2, \\
&p> 1+\dfrac{5}{n-1},\qquad q\ge 1+\dfrac{4}{n-1} \qquad \text{if}\ n=3,4,5.
\end{split}
\end{equation}
Then, we have the same conclusions as those in Theorem \ref{theorem3.1} and the estimates from \eqref {estimate3.1.1} to \eqref {estimate3.1.3} hold.
\end{theorem}

\begin{theorem}[Equation \eqref{equation3.3} with $j=1$] \label{theorem3.32}
Let $n=2,3,4$ and $\e$ is a sufficiently small positive constant. Assume that the following conditions for $p$ and $q$ hold:
\begin{equation} \label{exponent3.32}
p> 1+\dfrac{4}{n} \quad \text{ and }\quad q \ge 1+ \max\Big\{\frac{3}{n}, 1 \Big\}.
\end{equation}
Then, we have the same conclusions as those in Theorem \ref{theorem4.1} and the estimates from \eqref {estimate4.1.1} 
to \eqref {estimate4.1.4} hold.
\end{theorem}

\begin{proof}[Proof of Theorems \ref{theorem3.31} and \ref{theorem3.32}]
We introduce the solution space
$$ X(t) \equiv
\begin{cases}
X_1(t) &\text{ if }\, n=2,3,4 \, \text{ and }\, j=0, \\
X_2(t) &\text{ if }\, n=5 \, \text{ and }\, j=0, \\
Y(t) &\text{ if }\, n=2,3,4 \, \text{ and }\, j=1,
\end{cases} $$
where the spaces $X_1(t)$, $X_2(t)$ and $Y(t)$ appear as in the proof of Theorems \ref{theorem3.1} and \ref{theorem4.1}. The solutions to (\ref{equation3.3}) can be written by the following form:
\begin{align*}
u(t,x) &= u^{\text{ln}}(t,x) + \int_0^t \mathcal{K}_1(t-\tau,x) \ast_x \big(|\partial^j_t u(\tau,x)|^p+ a \circ \nabla |\partial^j_t u(\tau,x)|^q\big) d\tau \\ 
&=: u^{\text{ln}}(t,x)+ \bar{u}^{\text{nl}}(t,x).
\end{align*}
We define a mapping $\Psi: \,\, X(t) \longrightarrow X(t)$ in the following way:
$$ \Psi[u](t,x)= u^{\text{ln}}(t,x)+ \bar{u}^{\text{nl}}(t,x). $$
Our main is to indicate that the following pair of inequalities are fulfilled:
\begin{align*}
\|\Psi[u]\|_{X(t)}& \lesssim \|(u_0,u_1)\|_{\mathcal{A}}+ \|u\|^p_{X(t)}+ \|u\|^q_{X(t)}, \\
\|\Psi[u]- \Psi[v]\|_{X(t)}& \lesssim \|u-v\|_{X(t)} \big(\|u\|^{p-1}_{X(t)}+ \|v\|^{p-1}_{X(t)}+\|u\|^{q-1}_{X(t)}+ \|v\|^{q-1}_{X(t)}\big).
\end{align*}
Then, repeating the similar approach to we did in the proof of Theorems \ref{theorem3.1} and \ref{theorem4.1} we may arrive at the desired inequalities above, which are to finish the proof of Theorems \ref{theorem3.31} and \ref{theorem3.32}.
\end{proof}

\begin{remark}
\fontshape{n}
\selectfont
In terms of the admissible exponents for $p$ and $q$ in (\ref{exponent3.31}), here one recognizes that the effect of the nonlinear convection $a \circ \nabla (|\partial^j_t u|^q)$ is really remarkable in comparison with that of the usual power nonlinearities $|\partial^j_t u|^p$, where $j=0,1$. More precisely, we can say that the former nonlinearities brings some more flexibility for lower bounds than those coming from the latter nonlinearities.
\end{remark}


\section*{Acknowledgments}
This research of the first author (Tuan Anh Dao) is funded (or partially funded) by the Simons Foundation Grant Targeted for Institute of Mathematics, Vietnam Academy of Science and Technology. 
The work of the second author (H. TAKEDA) was supported in part by the Grant-in-Aid for Scientific Research (C) (No. 19K03596) from Japan Society for the Promotion of Science. The authors are grateful to the referee for his careful reading of the manuscript and for helpful comments.

\section*{Appendix}
This section is to provide several useful inequalities, which play a significant role in the proofs of Sections \ref{Semi-linear.1} and \ref{Semi-linear.2}.

\begin{proposition}[Fractional Gagliardo-Nirenberg inequality] \label{fractionalGagliardoNirenberg}
Let $1<r,\,r_0,\,r_1<\infty$, $\sigma >0$ and $s\in [0,\sigma)$. Then, it holds
$$ \|v\|_{\dot{H}^{s}_r}\lesssim \|v\|_{L^{r_0}}^{1-\theta}\, \|v\|_{\dot{H}^{\sigma}_{r_1}}^\theta, $$
where $\theta=\theta_{s,\sigma}(r,r_0,r_1)= \displaystyle\frac{\frac{1}{r_0}-\frac{1}{r}+\frac{s}{n}}{\frac{1}{r_0}-\frac{1}{r_1}+\frac{\sigma}{n}}$ and $\frac{s}{\sigma}\leq \theta\leq 1$.
\end{proposition}
For the proof one can see \cite{Ozawa}.

\begin{proposition}[Sobolev embedding] \label{SobolevEmbedding}
Let $\varepsilon >0$. Then, it holds
\[ \|v\|_{L^\ity} \lesssim \|v\|_{L^{2}}^{1-\theta}\, \|v\|_{\dot{H}^{\frac{n}{2}+\varepsilon}}^{\theta}, \quad \text{ where }\theta= \frac{\frac{n}{2}}{\frac{n}{2}+\varepsilon}. \]
\end{proposition}
The proof of Proposition \ref{SobolevEmbedding} is well-known. However, for the convenience of the reader, we will show it.
\begin{proof}
If $\|v\|_{L^{2}}=0$, we have $v=0$ a.e. in $\R^{n}$. Then, the statement is trivial.
Now we assume $\|v\|_{L^{2}} \neq 0$ and define 
\begin{equation*}
R= \left( 
\frac{\big\| \nabla^{\frac{n}{2} +\varepsilon } v\big\|_{L^{2}}}{ \|v\|_{L^{2}}}
\right)^{\frac{1}{\frac{n}{2} + \varepsilon}}.
\end{equation*}
It is easy to see 
\begin{equation*}
\begin{split}
\|v\|_{L^{\ity}} & = \| \mathcal{F}^{-1} (\widehat{v})\|_{L^{\ity}} \lesssim \|\widehat{v}\|_{L^{1}}= \int_{|\xi| \le R} |\widehat{v}(\xi)| d \xi+ \int_{|\xi| \ge R} |\widehat{v}(\xi)| d \xi=:A_{1} +A_{2}.
\end{split}
\end{equation*}
Now we apply H\"{o}lder inequality to have 
\begin{equation*}
\begin{split}
A_{1} \le \left( 
\int_{|\xi| \le R} d \xi
\right)^{\frac{1}{2}} \|\widehat{v}\|_{L^{2}} \lesssim R^{\frac{n}{2}} \|v\|_{L^{2}}
\end{split}
\end{equation*}
and
\begin{equation*}
\begin{split}
A_{2} \le \left( 
\int_{|\xi| \ge R} |\xi|^{-(n+2\varepsilon)} d \xi
\right)^{\frac{1}{2}} \big\| |\xi|^{\frac{n}{2}+\varepsilon} \widehat{v}\big\|_{L^{2}} \lesssim R^{-\varepsilon} \big\| \nabla^{\frac{n}{2}+\varepsilon} v\big\|_{L^{2}}.
\end{split}
\end{equation*}
Therefore, summing up the above we obtain 
\begin{equation*}
\begin{split}
\|v\|_{L^{\ity}} & \lesssim R^{\frac{n}{2}} \|v\|_{L^{2}}+ R^{-\varepsilon} \big\| \nabla^{\frac{n}{2}+\varepsilon} v\big\|_{L^{2}} \lesssim \|v\|_{L^{2}} \|^{1-\theta}\, \big\| \nabla^{\frac{n}{2}+\varepsilon} v\big\|_{L^{2}}^{\theta}, 
\end{split}
\end{equation*}
where $\theta= \frac{\frac{n}{2}}{\frac{n}{2}+\varepsilon}$. This is the desired estimate.
Hence, we have completed the proof of Proposition \ref{SobolevEmbedding}.
\end{proof}

Moreover, the following lemmas comes into play.
\begin{lemma} \label{LemmaIntegral.1}
Let $\alpha, \beta \in \R$. Then, the following inequality holds:
$$ \int_0^t (1+t-\tau)^{-\alpha}(1+\tau)^{-\beta}d\tau \lesssim
\begin{cases}
(1+t)^{-\min\{\alpha, \beta\}} &\text{ if }\, \max\{\alpha, \beta\}>1, \\
(1+t)^{-\min\{\alpha, \beta\}}\log(e+t)  &\text{ if }\, \max\{\alpha, \beta\}=1, \\
(1+t)^{1-\alpha-\beta} &\text{ if }\, \max\{\alpha, \beta\}<1.
\end{cases} $$
\end{lemma}
The proof of this lemma can be found in \cite{DaoReissig}.

\begin{lemma} \label{LemmaIntegral.2}
Let $c>0$, $0\le \alpha<1$ and $\beta \in \R$. Then, the following inequality holds:
$$ \int_0^t e^{-c(t-\tau)}(t-\tau)^{-\alpha} (1+\tau)^{-\beta}d\tau \lesssim (1+t)^{-\beta}. $$
\end{lemma}
For the ease of reading, we will prove this lemma even if it is standard and well-known.

\begin{proof}
Let us distinguish our consideration into two cases as follows:
\begin{itemize}[leftmargin=*]
\item If $t\in [0,1]$, then it is obvious that
$$ \int_0^t e^{-c(t-\tau)}(t-\tau)^{-\alpha} (1+\tau)^{-\beta}d\tau \lesssim \int_0^t (t-\tau)^{-\alpha} d\tau= \frac{1}{1-\alpha}t^{1-\alpha}\le \frac{1}{1-\alpha} \lesssim (1+t)^{-\beta}, $$
where we have used the condition $0\le \alpha<1$.
\item If $t\ge 1$, then we split the integral of the left-hand side into the following two parts:
\begin{align*}
I:&= \int_0^t e^{-c(t-\tau)}(t-\tau)^{-\alpha} (1+\tau)^{-\beta}d\tau \\ 
&= \int_0^{t/2} e^{-c(t-\tau)}(t-\tau)^{-\alpha} (1+\tau)^{-\beta}d\tau+ \int_{t/2}^t e^{-c(t-\tau)}(t-\tau)^{-\alpha} (1+\tau)^{-\beta}d\tau \\
&=: I_1+ I_2.
\end{align*}
Noticing that $t-\tau \in [t/2,t]$ for any $\tau\in [0,t/2]$ in the first integral one derives
\begin{align*}
I_1 \lesssim e^{-ct/2}t^{-\alpha}\int_0^{t/2} (1+\tau)^{-\beta}d\tau &\lesssim
\begin{cases}
e^{-ct/2}(1+t)^{1-\alpha-\beta} &\text{ if }\, \beta<1 \\
e^{-ct/2}(1+t)^{-\alpha}\log(1+t) &\text{ if }\, \beta=1 \\
e^{-ct/2}(1+t)^{-\alpha} &\text{ if }\, \beta>1
\end{cases} \\
&\lesssim (1+t)^{-\beta}.
\end{align*}
Thanks to the relation $1+\tau \approx 1+t$ for any $\tau\in [t/2,t]$, we may estimate the second integral in the following way:
\begin{align*}
I_2 &\lesssim (1+t)^{-\beta}\int_{t/2}^t e^{-c(t-\tau)}(t-\tau)^{-\alpha}d\tau \\
&\quad= (1+t)^{-\beta}\int_0^{t/2} e^{-c\rho}\rho^{-\alpha}d\rho \qquad (\text{by the change of variables }\rho= t-\tau) \\ 
&\quad= \frac{(1+t)^{-\beta}}{1-\alpha}\Big(e^{-c\rho}\rho^{1-\alpha}\Big|^{\rho=t/2}_{\rho=0}+ \frac{1}{c}\int_0^{t/2} e^{-c\rho}\rho^{1-\alpha}d\rho \Big) \\
&\lesssim \frac{(1+t)^{-\beta}}{1-\alpha}\Big(e^{-ct/2}t^{1-\alpha}+ \frac{1}{c}\int_0^{t/2} e^{-c\rho/2}d\rho \Big) \qquad (\text{by } 0\le \alpha<1) \\
&\lesssim (1+t)^{-\beta}.
\end{align*}
\end{itemize}
Combining all the above estimates leads to what we wanted to prove. Therefore, our proof is completed.
\end{proof}


\end{document}